\documentclass[11pt]{article}
\usepackage{mathtools}
\mathtoolsset{showonlyrefs}
\usepackage{amsmath}
\usepackage{amsthm}
\usepackage{amsfonts}
\usepackage{graphicx}
\usepackage{latexsym}
\usepackage{url}
\usepackage{bbm}
\usepackage{comment}

\newcommand{\R}{\mathbb{R}}

\newcommand{\E}{\mathbb{E}}
\newcommand{\PP}{\mathbb{P}} 
\newcommand{\N}{\mathbb{N}}
\newcommand{\Z}{\mathbb{Z}}

\newcommand{\C}{\mathbb C}
\newcommand{\1}{\mathbf{1}}
\renewcommand{\P}{\mathbb P}
\newcommand{\XX}{\mathbb X}

\newcommand{\cov}{\textrm{cov}} 
\newcommand{\var}{\textrm{var}}

\newcommand{\cX}{\mathcal X}

\def\ni{\noindent}

\newcommand{\schw}{\stackrel{d}{\longrightarrow}}

\newcommand{\bee}{\begin{equation}}
\newcommand{\eee}{\end{equation}}
\newcommand{\beea}{\begin{array}}
\newcommand{\eeea}{\end{array}}

\renewcommand{\theequation}{\arabic{section}.\arabic{equation}}

\theoremstyle{plain}
\newtheorem{prop}{Proposition}[section]
\newtheorem{cor}[prop]{Corollary}
\newtheorem{theo}[prop]{Theorem}
\newtheorem{lem}[prop]{Lemma}

\theoremstyle{definition}

\newtheorem{rem}[prop]{Remark}

\begin{document}

\title{A Berry--Esse\'en theorem for partial sums \\ of functionals  of heavy-tailed 
 moving averages}

\author{Andreas Basse-O'Connor\thanks{Department
of Mathematics, Aarhus University, Denmark,
E-mail: basse@math.au.dk}\;, \and 
Mark Podolskij\thanks{Department
of Mathematics, Aarhus University, Denmark, E-mail: mpodolskij@math.au.dk}\;, \and
Christoph Th\"ale\thanks{Faculty of Mathematics, Ruhr University Bochum,
Germany, E-mail: christoph.thaele@rub.de}}

\date{}

\maketitle

\begin{abstract}
In this paper we obtain Berry--Esse\'en bounds on partial sums of functionals 
of heavy-tailed moving averages, including the linear fractional stable noise,  stable fractional ARIMA processes and 
stable Ornstein--Uhlenbeck processes. Our rates are obtained  for   the Wasserstein and Kolmogorov distances, and depend
strongly on the interplay between the memory of the process, which is controlled by a parameter $\alpha$,  and its   tail-index, which  is controlled by a parameter $\beta$. In fact, we obtain 
 the classical $1/\sqrt{n}$ rate of convergence when the tails are not too heavy and the memory is not too strong, 
 more precisely, when  $\alpha\beta >3$ or $\alpha\beta>4$  in the case of Wasserstein and Kolmogorov distance, respectively. 

 Our quantitative  bounds rely on a new second-order Poincar\'e inequality on the Poisson space, which  we derive through a combination of Stein's method and  Malliavin calculus. This  inequality  improves and generalizes a result by  Last, Peccati, Schulte [\emph{Probab.\ Theory Relat.\ Fields}~\textbf{165} (2016)].

\ \

\bigskip


\end{abstract}

\section{Introduction}\label{sec1}
\setcounter{equation}{0}
\renewcommand{\theequation}{\thesection.\arabic{equation}}

\subsection{Overview}

The main goal of this  paper is to characterize the convergence rates associated with asymptotic normality of a class of statistics of L\'evy moving averages.  For processes with  finite  fourth moments, Theorem~8.2 in  \cite{LPS} obtains rates for a class of specific examples. Its proof relies  on  second-order  Poincar\'e inequalities on the Poisson space \cite[Theorem~1.1--1.2]{LPS}, which in turn are based on the celebrated Malliavin-Stein method.  The main novelty of our results and methodology is the study of convergence rates for processes having heavy tails and strong memory, as e.g.\ the linear fractional stable noise or fractional stable ARIMA processes.  In fact, in our setting the upper bounds in the  second-order  Poincar\'e inequalities   obtained in \cite{LPS} may converge to infinity after the application of our standard estimate \eqref{Dest} to them.   As a consequence, we 
develop  a new modified second-order Poincar\'e inequality on the Poisson space, which allows us to efficiently bound Wasserstein and Kolmogorov distances associated with normal approximation of a class of statistics of L\'evy moving averages. The improved bounds are important in their own right as they may prove to be useful in other contexts, where the considered stochastic process exhibits heavy tails and strong memory.

\subsection{Background}

The Berry--Esse\'en theorem gives a quantitative bound for the convergence rate in the classical central limit theorem. To be more precise, let  $(X_i)_{i\in \N}$ be a sequence of independent and identically distributed (i.i.d.)
random variables with mean zero, variance one and finite third moment, and set $V_n=n^{-1/2} \sum_{i=1}^n X_i$ and $Z\sim \mathcal N(0,1)$. Then, the central limit theorem says that $V_n\schw Z$ as $n\to \infty$, where $\schw$ denotes convergence in distribution.  A  natural next question is to ask 
for  quantitative bounds  between $V_n$ and $Z$, that is, how far is $V_n$ from $Z$ in a certain sense. An answer to this question is provided by the   Berry--Esse\'en theorem, which states that  
\begin{equation}\label{sdfljsdghs}
  d_K(V_n, Z):=\sup_{x\in \R} \big|\P(V_n\leq x)-\P(Z\leq x)\big|\leq C n^{-1/2}, 
\end{equation}
where $d_K$ denotes the Kolmogorov metric between two random variables and where $C$ is a constant depending on the third moment of the underlying random variables.  The Berry--Esse\'en bound
\eqref{sdfljsdghs} is optimal in the sense that there exist random variables as above such that $d_K(V_n,Z)$ is bounded from below by a constant times $n^{-1/2}$, see e.g.\ \cite[(5.26.2)]{Hoff}. 

The situation when the summands $(X_i)_{i\in \N}$ 
are dependent is much  more complicated, compared to the classical i.i.d.\ setting.  One of the most important models in this situation is the fractional Gaussian noise, which we will describe in the following. 
  For $H\in (0,1)$, the fractional Brownian motion is the unique centered Gaussian process $(Y_t)_{t\in\R}$ 
with covariance function 
\begin{equation}
 \textrm{cov}(Y_t,Y_u) = \frac{1}{2} \Big( |t|^{2H}+|u|^{2H}-|t-u|^{2H}\Big),\qquad \text{for all }t,u\in \R.  
\end{equation}
The fractional Gaussian noise $(X_n)_{n\in \Z}$ is the corresponding increment process $X_n=Y_{n}-Y_{n-1}$. 
Let 
\begin{equation}
  V_n=\frac{1}{\sqrt{n}} \sum_{j=1}^n \big( X_j^2-1\big) \qquad \text{and}\qquad v_n=\sqrt{\textrm{var}(V_n)}.
\end{equation}  
For $H<3/4$,  we have that $v_{n}^{-1} V_n \schw Z$ as $n\to \infty$. 
The first Berry--Esse\'en bound for the fractional Gaussian noise was obtained in Theorem~4.1 of  \cite{NouPec} and reads as 
\begin{equation}\label{sdlfjsdlfj-ueytegeg}
  d_K\Big(v_n^{-1}V_n,Z\Big)\leq C 
  \begin{cases}
  	n^{-1/2} \qquad & \text{if }H\in (0,1/2], \smallskip\\
  	n^{2H-3/2} & \text{if } H\in (1/2,3/4).
  \end{cases}
\end{equation}
In \eqref{sdlfjsdlfj-ueytegeg} we observe the phenomenon that for strong memory in $X$  (i.e.\ $H\in (1/2,3/4)$),
 we get a slower  rate of convergence.  Furthermore,  when    $H>3/4$, the memory in  $X$ 
is so strong that $V_n$ after proper normalization  converge to the Rosenblatt random variable in distribution, and hence 
has  a non-Gaussian fluctuation,  see e.g.\ Theorem~7.4.1 of \cite{Nou-Pec-book}. 

\subsection{Heavy-tailed  moving averages}

Let us now describe our results in more detail. We consider a two-sided L\'evy process $L=(L_t)_{t\in \R}$ with no Gaussian component, $L_0=0$ a.s.\ and L\'evy measure $\nu$, that is, 
for all $\theta\in \R$, the characteristic function of $L_1$ is given by 
\begin{equation}\label{sfldjsdf}
	\E[ e^{i \theta L_1}]=\exp\Big( \int_\R \big( e^{i\theta x } -1-i\theta x \chi(x)\big) \,\nu(dx)+i b \theta\Big), 
\end{equation}
where $b\in \R$, and $\chi$ is a truncation function, i.e.\  a bounded measurable
function such that $\chi(x) = 1+o(|x|)$ as $x \to  0$ and $\chi(x) = O(|x|^{-1})$ as $|x|\to \infty$.   We assume that the L\'evy measure $\nu$ has a density $\kappa$ satisfying 
\begin{equation}\label{sdfljsd}
	\kappa(x) \leq C  |x|^{-1-\beta} \qquad \text{for all }x\in \R\setminus\{0\}, 
\end{equation}
for  $\beta\in (0,2)$ and a  constant $C>0$.   
We consider a L\'evy moving average of the form
\begin{align} \label{X}
X_t= \int_{-\infty}^t g(t-s)\, dL_s,\qquad t\in\R,
\end{align}
where $g\!:\R\to \R$ is a measurable function such that the integral exists, see \cite{RajRos} for sufficient conditions. L\'evy moving averages are stationary infinitely divisible processes, and are often used to model long-range dependence and heavy tails. When the L\'evy process $L$ is symmetric, i.e.\ when $-L_1$ equals $L_1$ in distribution,  a sufficient condition 
for $X$ to be well-defined is that $\int_\R |g(s)|^\beta\,ds<\infty$, 
due to assumption \eqref{sdfljsd}.  Throughout the paper we will assume that the kernel function $g$ satisfies   
\begin{align} \label{gcond}
|g(x)| \leq K \big( x^\gamma\1_{\{0<x<1\}}+ x^{-\alpha}\1_{\{x\geq 1\}}\big), \qquad \text{for all }x>0,
\end{align} 
for some finite constants $K>0$, $\alpha>0$ and $\gamma\in \R$. 
We refer to Subsections~\ref{lsjdfljsdf}--\ref{ljsdflhsdfhg} for four important examples 
in this setting. 
The  statistics of interest are the partial sum functionals $V_n$ 
given by 
\begin{align} \label{vnf}
V_n = \frac{1}{\sqrt{n}} \sum_{t=1}^n \big( f(X_t) - \E[f(X_1)] \big),
\end{align}
based on a measurable function  $f\!:\R\to\R$ with $\E[|f(X_1)|]<\infty$.  Typical examples, which are important in statistics, are the empirical characteristic functions ($f\!: x\mapsto e^{i\theta x}$, where $\theta\in \R$),   the empirical distribution functions ($f\!:x\mapsto \1_{(-\infty,t]}(x)$, where $t\in \R$), and power variations ($f\!: x\mapsto |x|^p$, where $p>0$). 
For example, in a recent  paper \cite{MOP} the empirical characteristic function has been successfully employed to estimate the parameters of a linear fractional stable motion observed at high or low frequency. 

The major breakthrough on establishing central limit theorems for $V_n$ was achieved in the paper Hsing~\cite[Theorem~1]{Hsing}, 
and was extended in \cite{PT2, PT07, MOP, BLP, BHP18}, whereas  non-central limit theorems for $V_n$ are established in \cite{BHP18, BLP, Sur02, Sur04}.   From these results,  it  follows that if 
$(X_t)$ is given by \eqref{X} with   $L$ being a $\beta$-stable L\'evy process and  the
 kernel function $g$ satisfying \eqref{gcond} with $\gamma\geq 0$ and $\alpha\beta>2$ we have that 
\[
V_n \schw Z\sim\mathcal{N}(0, v^2) \qquad \text{with} \qquad v^2=\sum_{j \in \mathbb{Z}}\text{\rm cov}\left(f(X_0), f(X_{j})\right)\in [0,\infty),
\]
for all  bounded and measurable $f\!:\R\to \R$, cf.\ \cite[Theorem~2.1]{PT2}. 
For $\alpha\beta<2$, $V_n$ has a non-Gaussian fluctuation and a different scaling rate, see e.g.\ \cite[Theorem~1.2]{BLP}, and hence we will only consider the case $\alpha\beta >2$. 

%
%


\subsection{Main results}
 To present our main result let $C^2_b(\R)$ denote the space of twice continuously differentiable functions such that $f$, $f'$ and $f''$ are bounded. Our  result  reads as follows:

\begin{theo} \label{maintheo}
Let $(X_t)_{t\in \R}$ be a L\'evy moving average given by \eqref{X}, satisfying \eqref{sdfljsd} for some $0<\beta<2$, and \eqref{gcond} with  $\alpha\beta >2$ and  $\gamma >-1/\beta$.  Let $V_n$   
be the corresponding  partial sums of  functionals, given by  \eqref{vnf},  based on 
  $f\in C^2_b(\R)$.   Also, let $Z\sim\mathcal{N}(0,1)$ be a standard Gaussian random variable, and set $v_n= \sqrt{\textup{var}(V_n)}$ for all $n\in \N$. Then, $v_n\to v$, where $v\geq 0$ is given by 
  \begin{equation}\label{sdldfjsdf}
  v^2=\sum_{j \in \mathbb{Z}}\text{\rm cov}\left(f(X_0), f(X_{j})\right)
\end{equation}
and the series \eqref{sdldfjsdf} converges absolutely.  Suppose  that $v>0$. 
Then, $v_n^{-1}V_n\schw \mathcal N(0,1)$ as $n\to \infty$. Moreover, for each $n\in\N$, 
\begin{align} \label{wasserstein}
d_W\left( v_n^{-1}V_n, Z\right) \leq {}& C 
\begin{cases}
n^{-1/2} \qquad &\text{if } \alpha\beta >3,\\
n^{-1/2}\log(n) \qquad &  \text{if } \alpha\beta = 3,\\
n^{(2-\alpha\beta)/2} & \text{if } 2<\alpha\beta <3,
\end{cases}
\intertext{and}
\label{kolmogorov}
d_K\left( v_n^{-1}V_n, Z\right) \leq {}& C
\begin{cases}
n^{-1/2} & \text{if } \alpha\beta>4, \\
n^{-1/2} \log(n) \qquad  & \text{if } \alpha\beta=4, \\
n^{(2-\alpha\beta)/4} & 2<\alpha\beta<4,
\end{cases}
\end{align}
where $C>0$ is a constant that does not depend on $n$ and $d_W$ denotes the Wasserstein distance.
\end{theo}

\begin{rem} In the following we will make a few remarks on Theorem~\ref{maintheo}. 
	\begin{enumerate}
		\item The bounds on the Wasserstein and Kolmogorov distances to the normal distribution, depend on 
 the  interplay between memory of $X$, which is controlled by $\alpha$,  and the  tail-index of $X$, which 
 is controlled by  $\beta$. In fact, we obtain 
 the classical $1/\sqrt{n}$ rate of convergence when the tails are not too heavy and the memory is not too strong, 
 more precisely, when  $\alpha\beta >3$ or $\alpha\beta>4$  in the case of Wasserstein and Kolmogorov distance, respectively. 
 We conjecture that our bounds are optimal in this case. We note also  that all  rates in Theorem~\ref{maintheo} converge to zero. 
		\item Our main objective is to obtain the quantitative bounds 
		\eqref{wasserstein} and \eqref{kolmogorov}. However, we obtain 
		the limit theorem $V_n\schw \mathcal N(0,v^2)$ as a by-product, which is new whenever $L$ is not stable. 
		\item Our proof of Theorem~\ref{maintheo} relies on new second-order Poincar\'e inequalities, which provide general bounds between  Poisson functionals and Gaussian random variables in the Wasserstein and Kolmogorov distances,  
see Section~\ref{sec3}.   We believe that these inequalities,  which are  improvements  of  those   obtained in \cite{LPS}, are of independent interest.  Our new bounds are, in particular, important  in the regime of strong heavy tails combined with  strong memory, i.e.  
$\alpha\beta\in (2,3)$. 
In this setting, and  applying the estimate \eqref{Dest} on $D_z V_n$ (which we will  used throughout the paper), the upper bound from  \cite[Theorems~1.1 and 1.2]{LPS} diverges to infinity, and hence gives no information, whereas our new bounds converge to zero. 
	\end{enumerate}
\end{rem}

%
%
%
%

In the following   we will apply Theorem~\ref{maintheo} to the four important examples:  linear fractional stable noises, fractional L\'evy noises, stable fractional ARIMA processes,  and stable Ornstein--Uhlenbeck processes. Throughout we will fix the notation used in  Theorem~\ref{maintheo}, that is,  $V_n$ is given in  \eqref{vnf} with 
  $f\in C^2_b(\R)$,  $v_n= \sqrt{\textup{var}(V_n)}$,  $v^2$   given in \eqref{sdldfjsdf} satisfies $v^2>0$, and 
   $Z\sim\mathcal{N}(0,1)$ is  a standard Gaussian random variable.

\subsubsection{Linear fractional stable noises}\label{lsjdfljsdf}
Our first example concerns the linear fractional stable noise. To define this process 
we let  $L$ be a $\beta$-stable L\'evy process with $\beta\in (0,2)$, and 
\begin{equation}\label{sdfljsdlfj-jljsdfghsdfg}
  X_t=Y_t-Y_{t-1} \qquad \text{where}\qquad 
  Y_t = \int_{-\infty}^t \Big\{(t-s)_+^{H-1/\beta}  - (-s)_+^{H-	1/\beta}\Big\}\, dL_s,
\end{equation}
where $H\in (0,1)$. For $\beta=1$ we assume furthermore that $L$ is symmetric, that is, $L_1$ equals $-L_1$ in distribution. 
The linear fractional stable motion $(Y_t)_{t \in \R}$ has stationary increments and is self-similar with index $H$, and can be viewed as a heavy-tailed extension of the fractional Brownian motion, see \cite{SaromodTaqqu} for more details. In this setting we deduce that 
$\alpha=1-H+1/\beta$ and the condition $\alpha \beta>2$ translates to $\beta \in (1,2)$, $0<H<1-1/\beta$.  Since $\beta>1$ we never have
$\alpha \beta \geq 3$.

\begin{cor}	\label{sdlfjsdsdfsdflfj}
Let $(X_t)_{t\in\R}$ be the linear fractional stable noise defined as in \eqref{sdfljsdlfj-jljsdfghsdfg}.
For $\beta \in (1,2)$ and $0<H<1-1/\beta$ we have that 
\begin{equation}
  d_W(v_n^{-1}V_n,Z)\leq C n^{(1+\beta(H-1))/2}\qquad \text{and}\qquad  d_K(v_n^{-1}V_n,Z)\leq C n^{(1+\beta(H-1))/4},
\end{equation}
where $C>0$ is a constant not depending on $n$.
\end{cor}

\subsubsection{Linear fractional L\'evy noise}

In the following we will consider the case of a linear fractional L\'evy noise, 
which has higher moments compared to the  linear fractional stable noise.   Let $L$ be a mean zero L\'evy process  with a  L\'evy density $\kappa$ satisfying $\kappa(x)\leq |x|^{-1-\zeta}$ for all $x\in [-1,1]$, where $\zeta\in [0,2)$,  
and $\kappa(x)\leq |x|^{-3}$ for all $x\in \R$ with $|x|>1$. The assumptions on $L$ are e.g.\ satisfied for tempered stable 
L\'evy processes (with uniform tilting), cf.\ \cite{Ros-Tem-Sta}, and  ensures that $L$ has finite $r$-moments for all $r\in (0,2)$,  and that the Blumenthal--Getoor index $\beta_{BG}$ of $L$ satisfies $\beta_{BG}\leq \zeta$.  Let $(X_t)$ be given by 
\begin{equation}\label{sdfsdfsdfdsf}
  X_t = Y_t-Y_{t-1}, \qquad \text{where}\qquad Y_t=\int_{-\infty}^t \Big\{(t-s)^{-\rho}_+ -(-s)_+^{-\rho}\Big\}\,dL_s,
\end{equation}
where $\rho\in (0,  1/\zeta)$. We use the convention $1/0:=\infty$.  The above  assumptions ensures that both $X$ and $Y$ 
are well-defined stochastic processes. See \cite[Section~2.6.8]{PT17Book} or \cite{Mar} for more details. The assumptions  of Theorem~\ref{maintheo} are satisfied  for $\alpha=\rho+1$ and all $\beta\in [\zeta,2)$, and hence we obtain the following corollary:

\begin{cor}\label{slfdjsdhfghdgdggd}
Let $(X_t)$ be the linear fractional L\'evy noise defined in \eqref{sdfsdfsdfdsf} with $\zeta\in [0,2)$ and $\rho\in (0,1/\zeta)$. For all $\epsilon>0$ we have that 
\begin{align}
   {}&  d_W(v_n^{-1}V_n,Z)\leq C
   \begin{cases}
  n^{-1/2} \qquad & \text{if } \rho>\frac{1}{2}, \\
  n^{-\rho+\epsilon}  & \text{if } \rho\leq \frac{1}{2},	
  \end{cases}
\intertext{and} 
 {}&  d_K(v_{n}^{-1}V_n,Z)\leq C \begin{cases}
  	n^{-1/2}  \qquad &   \text{if }\rho>1, \\
  	n^{-\rho/2+\epsilon}    & \text{if }  \rho\leq 1. 
  \end{cases}
\end{align}
\end{cor}

\subsubsection{Stable fractional ARIMA processes}
In the following we will consider the stable fractional ARIMA process. To this end,  we let 
 $p,q\in \N$, and   $\Phi_p$ and $\Theta_q$ be   polynomials with real coefficients on the form 
\begin{align}
  \Phi_p(z)=1-\phi_1 z-\dots - \phi_p z^p,\qquad \text{and}\qquad \Theta_q(z)=1+\theta_1 z+\dots+\theta_q z^q,  
\end{align}
where we assume that $\Phi_p$ and $\Theta_q$ do not have common roots, and that $\Phi_p$ has no roots in the closed unit disk 
$\{z\in \C: |z|\leq 1\}$. The stable fractional ARIMA$(p,d,q)$ process $(X_n)_{n\in \N}$ is  the  solution 
to the equation
\begin{equation}\label{lsjdfljsdfsdfsdfsf}
  \Phi_p(B)X_n = \Theta_q(B)(1-B)^{-d} \epsilon_n 
\end{equation}
where $(\epsilon_n)_{n\in \N}$ are independent and identically symmetric $\beta$-stable random variables with $\beta\in (0,2)$, $B$ denotes the backshift operator, and $d\in \R\setminus \Z$. The equation should be understood as in  \cite[Section~2]{Kok-Taqqu}. For     $d<1-1/\beta$,  
 there exists a unique solution  $(X_n)_{n\in \N}$ to  \eqref{lsjdfljsdfsdfsdfsf}, and it is a discrete moving average of the form 
\begin{equation}
  X_n= \sum_{j=-\infty}^n b_{n-j} \epsilon_j
\end{equation}
for a certain sequence $(b_j)_{j\in \N}$ with $b_j\sim c_0 j^{d-1}$ as $j\to \infty$, where $c_0$ denotes a positive constant, cf.\   Theorem~2.1 of \cite{Kok-Taqqu}. Notice that the process $(X_n)_{n\in \N}$ can be written in distribution as
\[
X_n = \int_{-\infty}^n g(n-s)\, dL_s
\]
with $g(x)=\sum_{j\geq 0} b_j \1_{[j,j+1)} (x)$ and $L$ being a symmetric $\beta$-stable L\'evy process. Here $\alpha=1-d$ and
 by Theorem~\ref{maintheo} we obtain the following result. 

\begin{cor}\label{ljsdfjlsjdf}
	Let $(X_n)_{n \in \N}$ be the stable fractional ARIMA$(p,d,q)$ process given  by \eqref{lsjdfljsdfsdfsdfsf}, with $\beta\in (0,2)$ and $d\in \R\setminus \Z$ with $d<1-2/\beta$. Then, 
	\begin{align} \label{wasserstein-223}
d_W\left( v_n^{-1}V_n, Z\right) \leq  {}& C 
\begin{cases}
n^{-1/2} \qquad &\text{if } d <1-3/\beta,\\
n^{-1/2}\log(n) \qquad &  \text{if } d = 1-3/\beta,\\
n^{1-(1-d)\beta/2} & \text{if } d\in  (1-3/\beta,1-2/\beta),
\end{cases}
\intertext{and}
d_K\left( v_n^{-1}V_n, Z\right) \leq  {}& C 
\begin{cases}
n^{-1/2} \qquad &\text{if } d <1-4/\beta,\\
n^{-1/2}\log(n)\qquad &  \text{if } d = 1-4/\beta,\\
n^{(1-(1-d)\beta/2)/2} & \text{if } d\in (1-4/\beta,1-2/\beta),
\end{cases}
\end{align}
where $C>0$ is a constant not depending on $n$.
\end{cor}

\subsubsection{Stable Ornstein--Uhlenbeck processes}\label{ljsdflhsdfhg}
In our last example we will consider a stable Ornstein--Uhlenbeck process $(X_t)_{t\in\R}$, given by 
\begin{equation}\label{dsfdsdf}
  X_t=\int_{-\infty}^t e^{-\lambda(t-s)}\,dL_s, 
\end{equation}
where $L$ denotes  a $\beta$-stable L\'evy process with $\beta\in (0,2)$, and  $\lambda>0$ is a finite 
constant.  In this case $\alpha>0$ can be chosen arbitrarily large and we obtain the following result.

\begin{cor}\label{lsjdflhsdgdg}
Let $(X_t)_{t\in\R}$ be a stable Ornstein--Uhlenbeck process given by \eqref{dsfdsdf}. Then  
\begin{equation}
  d_W(v_n^{-1} V_n,Z)\leq C n^{-1/2} \qquad \text{and}\qquad     d_K(v_n^{-1}V_n,Z)\leq C n^{-1/2},
\end{equation}
where $C>0$ is a constant not depending on $n$.
\end{cor}

\subsection{Structure of the paper}

The paper is structured as follows. Section~\ref{sec2} presents a related result and some discussions. 
Basic notions of Malliavin calculus on  Poisson spaces and the new  bounds for the Wasserstein and Kolmogorov distances are demonstrated in Section~\ref{sec3}. In Section~\ref{sec4} we prove Theorem~\ref{maintheo} based on the general bounds obtained in Section~\ref{sec3}.

\section{Related literature and discussion} \label{sec2}
\setcounter{equation}{0}
\renewcommand{\theequation}{\thesection.\arabic{equation}}

Normal approximation of non-linear functionals of Poisson processes defined on general state spaces has become a topic of increasing interest during the last years. In particular, quantitative bounds for normal approximations were obtained by combining Malliavin calculus on the Poisson space with Stein's method. The resulting bounds have successfully been applied in various contexts such as stochastic geometry (see, e.g., \cite{DST,LachPecc,LachPecc2,LPS,LPST,ReitznerSchulte}), the theory of U-statistics (see, e.g., \cite{DST,DoeblerPeccati,ET,ReitznerSchulte}), non-parametric Bayesian survival analysis (see \cite{PeccPeruenster,PeccTaqqu}) or statistics of spherical point fields (see, e.g., \cite{BD,BDMP}). We refer the reader also to \cite{PeccatiReitznerBook}, which contains a representative collection of survey articles.

The first quantitative bounds for asymptotically normal functionals of L\'evy moving averages have been derived in \cite{LPS}. We briefly introduce their framework, but phrase their results in an equivalent way through L\'evy processes instead of Poisson random measures. 
Let $(X_t)_{t\in \R}$ denote a L\'evy moving average of the form \eqref{X} where $L$ is centered.   
	 Assume that the L\'evy measure $\nu$ satisfies the condition
\begin{align} \label{moments}
\int_{\R} |y|^j \,\nu(dy)<\infty \qquad \text{for } j\in\{1,2\},  
\end{align}
which implies that $(L_t)_{t\in \R}$ is of locally bounded variation with a finite second moment. Suppose that the kernel function $g$ satisfies the condition
\begin{align} \label{gcondition}
\int_{\R}  |g(x)| +g(x)^2 \,dx<\infty.
\end{align}  
The functional under consideration is defined by
\[
F_T = \int_{[0,T]} f(X_t ) \,dt,\qquad T>0,
\]
which can be interpreted as the continuous version of the statistic $V_n$. We now state Theorem~8.2 of \cite{LPS} in the case of $p=0$. 

\begin{theo} \label{LPSth}
\cite[Theorem 8.2]{LPS}  Suppose that conditions \eqref{moments} and
\eqref{gcondition} hold. Assume that 
\[
\text{\rm var} (F_T) \geq c T, \qquad T\geq t_0 
\] 
with $c,t_0>0$. Furthermore, suppose that $f \in C^2_b(\R)$.
Assume that $\int_{\R} |y|^{4} \,\nu(dy)<\infty$ and  $\E[|X_1|^{4}]<\infty$. Finally, assume that 
\[
 \int_{\R} \left(\int_{\R}  |g(y-x) g(y)| \,dy\right)^4 dx< \infty.
\] 
Let $Z$ be a standard Gaussian random variable. Then there exists a constant $C>0$ such that for all $T\geq t_0$, 
\[
d\left( \frac{F_T- \E[F_T]}{\sqrt{\text{\rm var} (F_T)}} , Z \right) \leq \frac{C}{\sqrt T},
\]
for $d=d_W$ and $d=d_K$. 
\end{theo}

While Theorem~\ref{LPSth} (and its proof) relies  heavily on a finite fourth moment,  
Theorem~\ref{maintheo} works  for infinite variance models, as e.g.\ stable processes. 
The heavy tails of these processes force us to introduce the improved version of the bounds in \cite[Theorems 1.1 and 1.2]{LPS} in the next section,  and they are also responsible for slower rates of convergence in 
Theorem~\ref{maintheo} compared to Theorem~\ref{LPSth}.

\section{New bounds for normal approximation on Poisson spaces}\label{sec3}
\setcounter{equation}{0}
\renewcommand{\theequation}{\thesection.\arabic{equation}}

The aim of this section is to introduce new bounds on the Wasserstein and the Kolmogorov distances between a Poisson functional and a standard Gaussian random variable.  Although similar bounds were previously derived in \cite{LPS,PSTU}, they are not sufficient in certain settings. For this reason, we shall provide an improved version, which is adapted to our needs. Since such a bound might be useful in other contexts as well, we formulate and prove it in a general set-up, which is specialised later in this paper. 

\subsection{Poisson spaces and Malliavin calculus}

We recall some basic notions of Malliavin calculus on Poisson spaces and refer the reader to \cite{LastPenroseBook,PeccatiReitznerBook} for further background information. We fix an underlying probability space $(\Omega,\mathcal{A},\P)$, let $(\XX,\cX)$ be a measurable space and $\lambda$ be a $\sigma$-finite measure on $\XX$ (in the applications we consider here $\XX$ will be of the form $\R\times\R$). By $\eta$ we denote a Poisson process on $\XX$ with intensity measure $\lambda$, see \cite{LastPenroseBook} for a formal definition and an explicit construction of such a process. We often consider $\eta$ as random element in the space of integer-valued $\sigma$-finite measures on $\XX$,  denoted by $\mathbf N$, which is  equipped with the $\sigma$-algebra generated by all evaluations $\mathbf \ni \mu\to \mu(A)$ for $A\in \cX$. 
A real-valued random variable $F$ is called a Poisson functional if there exists a measurable function $\phi:\mathbf N\to \R$ such that $\PP$-almost surely $F=\phi(\eta)$. We let $L^2_\eta$ denote the space of all square-integrable Poisson functionals $F=\phi(\eta)$. 

It is well known that each $F\in L_\eta^2$ admits a chaotic decomposition 
\begin{align} \label{chaos}
F=\E F + \sum_{m=1}^{\infty} I_m(f_m),
\end{align}
where $I_m$ denotes the $m$th order integral with respect to the compensated  (signed) measure $\hat{\eta}=\eta - \lambda$ and $f_m:\XX^m \to \R$
are symmetric functions with $f_m \in L^2(\lambda^m)$. Here, for a measure $\mu$ on $\XX$ and $k\in\N$ we write $L^2(\mu^k)$ for the space of functions $f:\XX^k\to\R$, which are square integrable with respect to the $k$-fold product measure of $\mu$, and $\|\,\cdot\,\|_{L^2(\mu^k)}$ for the corresponding $L^2$-norm. For $z\in\XX$ and a Poisson functional $F=\phi(\eta) \in L_\eta^2$ we denote by
$$
D_zF = \phi(\eta+\delta_z) - \phi(\eta)
$$
the Malliavin derivative of $F$ in direction $z$, also known as the difference operator. Here, $\delta_z$ stands for the Dirac measure at $z\in\XX$. We can consider $DF$ as a function on $\Omega\times\XX$ acting as $(\omega,z)\mapsto D_zF(\omega)$. If $DF$ is square integrable with respect to the product measure $\PP\otimes\lambda$ we shall write $DF\in L^2(\PP\otimes\lambda)$ in what follows. Finally, let us define the second-order Malliavin derivative of $F$ with respect to two points $z_1,z_2\in\XX$ by putting
\begin{align*}
D^2_{z_1,z_2}F &:=D_{z_1}(D_{z_2}F) = \phi(\eta+\delta_{z_1}+\delta_{z_2}) - \phi(\eta+\delta_{z_1}) - \phi(\delta_{z_2}) + \phi(\eta)
\end{align*}
(note that this definition is symmetric in $z_1$ and $z_2$). 

The Kabanov-Skorohod integral $\delta$ maps random functions $u$ from $L^2_{\eta}(\PP\otimes\lambda)$ to  random variables in 
$L_\eta^2$. To introduce the definition of the operator $\delta$, let 
\[
u(z)= \sum_{m=0}^{\infty} I_m\left(h_m(z,\,\cdot\,) \right), \qquad z\in \XX,
\] 
denote the chaos expansion of $u(z)$, where  $h_m: \XX^{m+1} \to \R$ are measurable functions. The domain $\text{dom}\,\delta$ of 
$\delta$ consists of all random functions $u$ that satisfy the condition
\[
\sum_{m=0}^{\infty} (m+1)! \|\tilde{h}_m\|_{L^2(\lambda^{m+1})}^2 < \infty,
\]
where $\tilde{h}_m$ denotes the symmetrisation of the function $h_m$.  For $u \in \text{dom}\,\delta$  the Kabanov-Skorohod integral is defined by
\[
\delta(u)= \sum_{m=0}^{\infty} I_{m+1}(\tilde{h}_m).
\]   
Finally, we introduce the Ornstein-Uhlenbeck generator $L$ and its (pseudo) inverse $L^{-1}$. The domain $\text{dom}\, L$ of $L$ consists of all elements
$F\in L_\eta^2$ with chaotic decomposition \eqref{chaos} that additionally satisfy the condition
\[
\sum_{m=1}^{\infty} m^2 m! \|f_m\|_{L^2(\lambda^{m})}^2 < \infty.
\] 
For $F \in \text{dom} \,L$ with chaotic decomposition \eqref{chaos}   we define
\[
LF= -  \sum_{m=1}^{\infty} mI_m(f_m), \qquad   L^{-1}F= -  \sum_{m=1}^{\infty} \frac 1m I_m(f_m).
\]
The important relationships between the introduced operators can be summarised as follows:
\begin{align*}
\text{(i) }& LL^{-1}F = F, \qquad\text{(if $F$ is a centred Poisson functional)}\\[1.5 ex]
\text{(ii) }& LF = -\delta DF, \qquad \text{(if $F \in \text{dom}\, L$)}\\
\text{(iii) }& \E[F\delta(u)] = \E\int \left(D_zF\right)\,u(z)\,\lambda(dz),\qquad\text{(integration-by-parts)}
\end{align*}
where $u \in \text{dom}\, \delta$. We refer to \cite{PeccatiReitznerBook} or \cite{PSTU} for a more detailed exposition.

\subsection{Wasserstein distance}\label{sec2.1}

In this subsection we derive quantitative bounds for the Wasserstein distance. We recall that the Wasserstein distance between two random variables $F$ and $G$ is defined by 
\[
d_W(F,Z) := \sup_{h\in\textup{Lip}(1)}\big|\E[h(F)]-\E[h(G)]\big|,
\]
where the supremum is running over all Lipschitz functions $h:\R\to\R$ with a Lipschitz constant less than or equal to $1$. To formulate the next result
we introduce the three quantities
\begin{align}\label{sldfjlsdjfljs}
\gamma_1^2 &:= 4  \int  \E[(D_{z_1} F)^2(D_{z_2} F)^2]^{1/2} \,
\E[(D_{z_1,z_3}^2 F)^2(D_{z_2,z_3}^2 F)^2]^{1/2} \,\lambda^{3} (d(z_1,z_2,z_3)), \\[1.5 ex]
\label{ljsdfghksdhf}
\gamma_2^2 &:= \int 
\E[(D_{z_1,z_3}^2 F)^2(D_{z_2,z_3}^2 F)^2] \,\lambda^{3} (d(z_1,z_2,z_3)) , \\[1.5 ex]
\gamma_3 &:=  \int  \E[|D_{z} F|^3]^{1/3} \,
 \E[\min(\sqrt{8}|D_{z} F|^{3/2}, |D_{z} F|^3)]^{2/3}\, \lambda (d z)
 \label{sdlfjsdlfh}
\end{align} 
(although $\gamma_1$, $\gamma_2$ and $\gamma_3$ depend on the Poisson functional $F$, we suppress this dependency in our notation for simplicity). The theorem below  is an improved  version of the second-order Poincar\'e inequality for Poisson functionals from \cite[Theorem 1.1]{LPS}, where the main difference stems from the  term $\gamma_3$. 

\begin{theo}\label{theo:SecondOrderPoincare}
Let $F\in L_\eta^2$ be a Poisson functional, which satisfies $DF\in L^2(\P\otimes\lambda)$, $\E[F]=0$ and $\E[F^2]=1$. Further, let $Z\sim\mathcal{N}(0,1)$ be a standard Gaussian random variable. Then,
$$
d_W(F,Z) \leq \gamma_1+\gamma_2+\gamma_3.
$$
\end{theo}

\begin{rem} \label{rem3} \rm
Our bound $\gamma_3$ improves the corresponding quantity in \cite[Theorem 1.1]{LPS}, which one can obtain by replacing the term 
$\E[\min(\sqrt{8}|D_{z} F|^{3/2}, |D_{z} F|^3)]$ in $\gamma_3$  through $\E[ |D_{z} F|^3]$. It turns out that our improvement is absolutely 
crucial as the quantity introduced in \cite[Theorem 1.1]{LPS} converges to infinity in our setting. 

In our framework the main problem appears when the term $|D_{z} V_n|$ becomes large, say, larger than $1$. Such an event has a relatively high weight under the measure $\lambda$. On the contrary, in the setting of  Theorem \ref{LPSth} large values of 
$|D_{z} V_n|$ have a very low probability under higher moment conditions on the L\'evy measure imposed in Theorem~\ref{LPSth}. 
Hence, the main improvement of the original bound introduced in \cite[Theorem 1.1]{LPS} stems from carefully distinguishing between large and small values of $|D_{z} V_n|$. \qed
\end{rem}

\noindent
Apart from the application of Malliavin calculus
another ingredient on which the proof of Theorem \ref{theo:SecondOrderPoincare} is build is Stein's method for normal approximation (we refer to \cite{ChenShao} for a general account). The starting point is the following observation: a random variable $Z$ has a standard Gaussian distribution $\mathcal{N}(0,1)$ if and only if
$$
\E[f'(Z)-Zf(Z)] = 0
$$
for every continuous and piecewise continuously differentiable function $f$ with $\E[|f'(Z)|]<\infty$. This characterization together with the definition of the Wasserstein distance motivates to consider for given Borel function $h:\R\to\R$ with $\E[|h(Z)|]<\infty$ the first-order differential equation
\begin{equation}\label{sldfjhsdlhf}
 h(x) - \E[h(Z)] = f'(x) - xf(x)\,,\qquad x\in\R,
\end{equation}
also known as the Stein equation for normal approximation. A solution to \eqref{sldfjhsdlhf} is an absolutely continuous function 
$f:\R\to \R$ such that there exists a version of the derivative $f'$ satisfying \eqref{sldfjhsdlhf} for all $x\in \R$. From \cite{ChenShao} it is known that for given $h\in\textup{Lip}(1)$ the (unique bounded) solution $f=f_h$ of this equation satisfies $f\in C^1$ with $f'$ absolutely continuous and 
$$
\| f\|_\infty \leq 2,\qquad \|f'\|_\infty \leq 1,\qquad\text{and}\qquad\|f''\|_\infty \leq 2,
$$
where we write $\|\,\cdot\,\|_\infty$ for the supremum norm. Defining $\mathcal{F}_W$ to be the class of functions satisfying these two constraints, replacing $x$ by $F$ in the Stein equation and taking expectations on both sides we arrive at
$$
d_W(F,Z) \leq \sup_{f\in\mathcal{F}_W}\big|\E[f'(F)-Ff(F)]\big|.
$$
Starting with this estimate, we can now present the proof of Theorem \ref{theo:SecondOrderPoincare}.

\begin{proof}[Proof of Theorem \ref{theo:SecondOrderPoincare}]
Let $f\in\mathcal{F}_W$ and fix $a,b\in\R$. Then using the bound for $f'$ we observe that
\begin{align*}
|f(b)-f(a)-f'(a)(b-a)| & \leq |f(b)-f(a)| +|f'(a)||b-a| \leq 2\|f'\|_\infty |b-a| \leq 2|b-a|.
\end{align*}
Similarly, using the bound for $f''$ we have, by Taylor approximation,
\begin{align*}
|f(b)-f(a)-f'(a)(b-a)| & \leq \frac{\|f''\|_\infty}{ 2}(b-a)^2 \leq (b-a)^2
\end{align*}
and so
\begin{align*}
|f(b)-f(a)-f'(a)(b-a)| \leq \min(2|b-a|,(b-a)^2).
\end{align*}
Next, using the definition of the Malliavin derivative and a Taylor expansion of $f$ around $F$ we see that
$$
D_zf(F) = f'(F)(D_zF) + R(D_zF),\qquad z\in\XX,
$$
where in view of the above considerations the remainder term $R(\cdot)$ satisfies the estimate $|R(y)|\leq\min(2|y|,y^2)$ for all $y\in\R$. Applying this together with the three relations for the Malliavin operators presented in the previous subsection, we see that
\begin{align*}
\E[Ff(F)] &= \E[LL^{-1}Ff(F)] = -\E[\delta(DL^{-1}F)\,f(F)] = \E\int (D_zf(F))(-D_zL^{-1}F)\,\lambda(dz)\\
&=\E\Big[f'(F)\int (D_zF)(-D_zL^{-1}F)\,\lambda(dz)\Big] + \E\int R(D_zF)(-D_zL^{-1}F)\,\lambda(dz).
\end{align*}
As a consequence,
\begin{align*}
|\E[f'(F)-Ff(F)]| &\leq \Big|\E\Big[f'(F)\Big(1-\int (D_zF)(-D_zL^{-1}F)\,\lambda(dz)\Big)\Big]\Big|\\
&\qquad\qquad\qquad + \E\int\min(2|D_zF|,(D_zF)^2)\,|D_zL^{-1}F|\,\lambda(dz)\\
&\leq \E\left[\Big|1-\int (D_zF)(-D_zL^{-1}F)\,\lambda(dz)\Big| \right]\\
&\qquad\qquad\qquad + \E\int\min(2|D_zF|,(D_zF)^2)\,|D_zL^{-1}F|\,\lambda(dz),
\end{align*}
where we used that $\|f'\|_\infty\leq 1$. Next, we apply \cite[Proposition 4.1]{LPS} to conclude that
\begin{equation}\label{hdgdgeeue}
 \E\left[\Big|1-\int (D_zF)(-D_zL^{-1}F)\,\lambda(dz)\Big| \right] \leq \gamma_1+\gamma_2.
\end{equation}
For the term $\gamma_3$ we use H\"older's inequality with H\"older conjugates $3$ and $3/2$. Together with \cite[Lemma 3.4]{LPS}, which in our situation says that $\E[|D_zL^{-1}F|^3]\leq\E[|D_zF|^3]$ for all $z\in\XX$, this leads to
\begin{align*}
&\E\int\min(2|D_zF|,(D_zF)^2)\,|D_zL^{-1}F|\,\lambda(dz)\\
&\leq \int(\E[|D_zL^{-1}F|^3])^{1/3}(\E[\min(2|D_zF|,(D_zF)^2)^{3/2}])^{2/3}\,\lambda(dz)\\
&\leq \int(\E[|D_zF|^3])^{1/3}(\E[\min(\sqrt{8}|D_zF|^{3/2},|D_zF|^3)])^{2/3}\,\lambda(dz)=\gamma_3.
\end{align*}
The proof is thus complete.
\end{proof}

\subsection{Kolmogorov distance} \label{sec2.2} 

Now we turn our attention to the Kolmogorov distance between the two random variables $F$ and $Z$, which is defined as
\[
d_K(F, Z) = \sup_{x \in \R} \left| \mathbb{P}(F\leq x) -  \mathbb{P}(Z \leq x)\right|. 
\]
Let $f=f_x$ be the bounded solution of the Stein's equation associated with the function $1_{(-\infty,x]}$ for a given $x\in\R$. It is well known that this solution satisfies 
the inequalities 
\begin{equation}\label{sdlfjsdgfsdfsdf}
 \|f\|_{\infty} \leq \frac{\sqrt{2 \pi}}{4} \qquad \text{and} \qquad \|f'\|_{\infty}\leq 1
\end{equation}
(we interpret $f'$ as the left-sided derivative at the point $x$, where $f$ is not differentiable). Hence, with $\mathcal F_K$ denoting the class of all absolutely continuous functions satisfying \eqref{sdlfjsdgfsdfsdf} 
 we have 
\begin{equation}
d_K(F,Z)\leq \sup_{f\in \mathcal F_K } |\E[ f'(F)-F f(F)]| 
\end{equation}
where  $Z\sim \mathcal N(0,1)$.
To obtain modified bounds for the Kolmogorov distance we introduce an arbitrary measurable function $\varphi: \R \to \R$ such that $0 \leq \varphi \leq 1$. We may simply use $\varphi=1_{[-1,1]}$ or a smooth function with compact support. As discussed in Remark \ref{rem3} we will use the function $\varphi$  to distinguish between large and small values of 
the quantity $|D_{z} V_n|$.

To formulate the analogue of Theorem~\ref{theo:SecondOrderPoincare} for the Kolmogorov distance we introduce the quantities
\begin{align}
\label{sldfjsldjhf}
\overline{\gamma}_3 &:= 2 \int \E[(D_zF)^2 (1- \varphi(D_zF))^2]^{1/2}  \E[ |D_z F|^2]^{1/2}  \, \lambda(dz), \\[1.5 ex]
\gamma_4 &:= \left( \frac{1}{2} (\E F^4)^{1/4} + \frac{\sqrt{2 \pi}}{8} \right)
\int \E[(D_zF)^4  \varphi(D_zF)^2]^{1/2}  \E[ |D_z F|^4]^{1/4}   \,\lambda(dz), \\[1.5 ex]
\gamma_5 &:= \sqrt{\int \E[\varphi(D_zF)^2 (D_zF)^4]^{1/2} \E[(D_zF)^4]^{1/2} 
\, \lambda(dz)}, \\[1.5 ex]
 \gamma_6^2 &:= 3 \int \Big(  \E[ (D_{z_2}(\varphi(D_{z_1}F) D_{z_1}F))^4]^{1/2} \E[|D_{z_1} F|^4]^{1/2} \\[1.5 ex]
&\qquad\qquad+ \E[\varphi(D_{z_1}F)^2 (D_{z_1}F)^4 ]^{1/2} \E[|D^2_{z_2,z_1} F|^4]^{1/2}\\[1.5ex]
&\qquad\qquad +  \E[ (D_{z_2}(\varphi(D_{z_1}F) D_{z_1}F))^4]^{1/2} \E[|D^2_{z_2,z_1} F|^4]^{1/2}\Big)  \,\lambda^{ 2}(d(z_1,z_2))
\label{sldfjsdgf}
\end{align}
(again, we suppress in our notation the dependency on the Poisson functional $F$).

We are now prepared to present our general estimate for the Kolmogorov distance between a Poisson functional and a standard Gaussian random variable.

\begin{theo} \label{Kdistance} 
Let $F\in  L_\eta^2$ be a Poisson functional which satisfies $DF\in L^2(\PP\otimes\lambda)$, $\E[F]= 0$ and $\E[F^2]= 1$. Further, let $Z\sim\mathcal{N}(0,1)$ be a standard Gaussian random variable. Then,
\[
d_K(F, Z) \leq \gamma_1+\gamma_2+\overline{\gamma}_3 + \gamma_4+\gamma_5+\gamma_6.
\]
\end{theo}

\begin{rem} \rm 
Theorem \ref{Kdistance} improves the original bounds from  \cite[Theorem 1.2]{LPS}, which correspond to the choice $\varphi=1$. 
We recall once again that the bound from \cite[Theorem~1.2]{LPS} converges to infinity in our setting, and thus the improvement in
Theorem \ref{Kdistance} is absolutely crucial for the proof of Theorem~\ref{maintheo}. \qed
\end{rem}
\begin{proof}
As in the proof of Theorem \ref{theo:SecondOrderPoincare} we need to bound the term 
\begin{align*}
\E[f'(F) - Ff(F)] &= \E\Big[f'(F)\Big(1- \int (D_zF) (-D_zL^{-1} F)\,\lambda(dz)\Big)\Big] \\[1.5 ex]
& \qquad\qquad- \E \left[ \int (-D_zL^{-1} F)\int_0^{D_zF} \{f'(F+t) -f'(F) \}\,dt\,\lambda(dz)    \right],
\end{align*}
where the decomposition is  derived in \cite[Equation (3.4)]{ET}. We have already seen in \eqref{hdgdgeeue} that the inequality 
\begin{align} \label{gamma12}
\left| \E\Big[f'(F)\Big(1- \int (D_zF) (-D_zL^{-1} F)\,\lambda(dz)\Big)\Big] \right| \leq \gamma_1 + \gamma_2
\end{align}
holds, where we have used that $\| f\|_\infty\leq 1$. To treat the second term of the decomposition we need to distinguish whether $D_zF$ takes small or large values. 
In particular, we have that 
\begin{align} \label{firstbound}
&\left| \E \left[ \int (-D_zL^{-1} F)\int_0^{D_zF} \{f'(F+t) -f'(F) \}\,dt\,\lambda(dz)    \right]\right|\\[1.5 ex]
&\leq 2 \E \int |D_zF| (1- \varphi(D_zF)) |D_zL^{-1} F|\,\lambda(dz) \nonumber \\[1.5 ex]
& \qquad\qquad+  \E \left[ \int |D_zL^{-1} F|\varphi(D_zF) \left|\int_0^{D_zF} \{f'(F+t) -f'(F) \}\,dt \right|\,\lambda(dz) \right].
\end{align}  
Now, repeating the methods  from the proof of \cite[Theorem 3.1]{ET} (see pages 7--9 therein) 
for the last term in \eqref{firstbound}, we conclude that 
\begin{align} \label{newbound1}
&\left| \E \left[ \int (-D_zL^{-1} F)\int_0^{D_zF} \{f'(F+t) -f'(F) \}\,dt\,\lambda(dz)    \right]\right| \nonumber\\[1.5 ex]
&\qquad\leq 2 \E \int |D_zF| (1- \varphi(D_zF)) |D_zL^{-1} F|\,\lambda(dz) \nonumber \\[1.5 ex]
&\qquad\qquad+ \frac{\sqrt{2 \pi}}{8} \E\int  \varphi(D_zF) (D_zF)^2|D_zL^{-1} F|\,\lambda(dz)\nonumber\\[1.5 ex]
&\qquad\qquad+ \frac{1}{2} \E\int  \varphi(D_zF) (D_zF)^2 |F\times D_zL^{-1} F|\,\lambda(dz) \nonumber \\[1.5 ex]
&\qquad\qquad+ \sup_{x \in \R} \E\int  \varphi(D_zF) (D_zF) D_z(1_{\{F>x\}}) |D_zL^{-1} F|\,\lambda(dz).
\end{align}
In the next step, we need to apply the ideas from \cite{LPS} to the new bound at \eqref{newbound1}.
We obtain that 
\begin{align*}
&2 \E \int |D_zF| (1- \varphi(D_zF)) |D_zL^{-1} F|\,\lambda(dz) \\[1.5 ex]
&\qquad\qquad\leq 2
\int \E[(D_zF)^2 (1- \varphi(D_zF))^2]^{1/2}  \E[ |D_zL^{-1} F|^2]^{1/2}   \lambda(dz) \\[1.5 ex]
& \qquad\qquad\leq 2 \int \E[(D_zF)^2 (1- \varphi(D_zF))^2]^{1/2}  \E[ |D_z F|^2]^{1/2}  \, \lambda(dz) =\overline{\gamma}_3. 
\end{align*}
Similarly, we have that 
\begin{align*}
& \frac{\sqrt{2 \pi}}{8} \E\int  \varphi(D_zF) (D_zF)^2|D_zL^{-1} F|\,\lambda(dz)\\[1.5 ex]
&\qquad\leq  \frac{\sqrt{2 \pi}}{8} \int \E[(D_zF)^4  \varphi(D_zF)^2]^{1/2}  \E[ |D_z F|^2]^{1/2}\,   \lambda(dz)  =: \gamma_{4}^{(1)},
\end{align*}
and 
\begin{align*}
& \frac{1}{2} \E\int  \varphi(D_zF) (D_zF)^2 |F\times D_zL^{-1} F|\,\lambda(dz)\\[1.5 ex]
&\qquad\leq   \frac{1}{2}(\E F^4)^{1/4}
\int \E[(D_zF)^4  \varphi(D_zF)^2]^{1/2}  \E[ |D_z F|^4]^{1/4}  \, \lambda(dz)  =: \gamma_{4}^{(2)}.
\end{align*}
Moreover, we note that
$$
\gamma_{4}^{(1)} + \gamma_{4}^{(2)} \leq \gamma_{4}.
$$
Next, we treat the last term in \eqref{newbound1}. We set $g(z)= 
\varphi(D_zF) (D_zF)  |D_zL^{-1} F|$ and observe the inequality
\[
\E\int D(1_{\{F>x\}}) g(z) \,\lambda(dz)= \E[1_{\{F>x\}} \delta(g)] \leq \E[ \delta^2(g)]^{1/2}
\]
is valid, where we recall that $\delta(g)$ stands for the Kabanov-Skorohod integral of $g$.
Furthermore, we have the Kabanov-Skorohod isometric formula
\[
\E[ \delta^2(g)] = \E\int g^2(z)\, \lambda(dz) + \E\int \int (D_yg(z))^2 \lambda(dy)\, \lambda(dz) =: A_1 + A_2,
\]
see \cite[Theorem 5]{LastChapter}.
Applying the Cauchy-Schwarz inequality we deduce that
\[
A_1 \leq \int \E[\varphi(D_zF)^2 (D_zF)^4]^{1/2} \E[(D_zF)^4]^{1/2} \,
 \lambda(dz) = \gamma_5^2.
\]
For the last term $A_2$ we conclude, using the inequality
$$
|D_y(GH)| \leq |H D_y G| + |G D_y H| + |D_yH D_y G|,
$$
that
\begin{align*}
A_2 &\leq 3 \int  \Big( \left(D_{z_2}(\varphi(D_{z_1}F) D_{z_1}F) \right)^2 |D_{z_1}L^{-1} F|^2
+ \varphi(D_{z_1}F) (D_{z_1}F)^2 |D^2_{z_2,z_1}L^{-1} F|^2 \\[1.5 ex]
& \qquad\qquad\qquad+ \left(D_{z_2}(\varphi(D_{z_1}F) D_{z_1}F) \right)^2 |D^2_{z_2,z_1}L^{-1} F|^2 \Big) \,\lambda^{ 2}(d(z_1,z_2))
\\[1.5 ex]
& \leq 3 \int   \Big( \E[ (D_{z_2}(\varphi(D_{z_1}F) D_{z_1}F))^4]^{1/2} \E[|D_{z_1} F|^4]^{1/2}\\[1.5 ex]
&\qquad\qquad\qquad + \E[\varphi(D_{z_1}F)^2 (D_{z_1}F)^4 ]^{1/2} \E[|D^2_{z_2,z_1} F|^4]^{1/2} \\[1.5 ex]
& \qquad\qquad\qquad +  \E[ (D_{z_2}(\varphi(D_{z_1}F) D_{z_1}F))^4]^{1/2} \E[|D^2_{z_2,z_1} F|^4]^{1/2} \Big) \,
\lambda^{ 2}(d(z_1,z_2))\\
&=\gamma_6^2.
\end{align*}
Combining \eqref{gamma12} and \eqref{newbound1}, we conclude the assertion of Theorem~\ref{Kdistance}. 
\end{proof}

\section{Proof of Theorem~\ref{maintheo}}\label{sec4}
\setcounter{equation}{0}
\renewcommand{\theequation}{\thesection.\arabic{equation}}

All positive constants, which do not depend on $n$, are denoted by $C$ although they may change from occasion to occasion. Furthermore, we assume without loss of generality that $K=1$ in condition \eqref{gcond}.  We extend the definition of the kernel $g$ to the whole real line by setting $g(x)=0$ for $x\leq 0$.  To apply Theorems~\ref{theo:SecondOrderPoincare} and \ref{Kdistance} it will be   useful for us to represent the process $(X_t)_{t \in \R}$, in \eqref{X},  in terms of an integral with respect to a Poisson random measure. 
Namely, if $\eta$ denotes a Poisson random measure on $\R^2$ constructed from $L$, see e.g.\ \cite[Theorem~19.2]{Sato}, 
then $\eta$ has intensity measure $\lambda$ given by 
$$
\lambda(ds,dy) =ds\,\nu(dy),
$$
where $\nu$ is the L\'evy measure of $L$,  defined in \eqref{sfldjsdf}.  We can re-write
 $X_t=\int_{-\infty}^t g(t-s)\,dL_s$ as 
\begin{equation}\label{sdklfhjsdgf}
	X_t= \int_{\R^2} g(t-s)x\, \Big( \eta(ds,dx)-\chi(g(t-s)x)\,ds\,\nu(dx)\Big) +\tilde \theta
\end{equation}
where the integral is defined as in \cite[p.~3236]{Ro18}, and 
\begin{equation}\label{skdhfksdgfgss}
\tilde \theta= \int_\R \Big[g(s) \theta + \int_\R \Big( \chi(xg(s))-g(s)\chi(x)\Big) \,\nu(dx)\Big] \,ds.
\end{equation}
The integrals in \eqref{sdklfhjsdgf} and \eqref{skdhfksdgfgss} exist since $X_t$ is well-defined, cf.\ \cite[Theorem~2.7]{RajRos}. 
By \eqref{sdklfhjsdgf}  it follows that  $X_t$ is a Poisson functional for all $t\in \R$, i.e.\ there exists 
a measurable mapping $\phi=\phi_t:\textbf{N}\to\R$  such that $X_t=\phi(\eta)$. Throughout this section we will repeatedly 
use that for any measurable positive function $f:\R^2\to \R_+$ we have that 
\begin{equation}
  \int_{\R^2} f(z)\,\lambda(dz)\leq C \int_\R \Big( \int_\R f(s,x)\,ds\Big) |x|^{-1-\beta}\,dx,
\end{equation}
which follows by assumption \eqref{sdfljsd} on $\nu$. Here and below, $C$ will denote a strictly positive and finite constant whose value might change from occasion to occasion.

\subsection{Preliminary estimates}

We let $z=(x,s)\in\R^2$, $z_j=(x_j,s_j) \in \R^2$ for $j\in\{1,2,3\}$, and  $f\in C^2_b(\R)$. By the mean-value theorem and \eqref{sdklfhjsdgf}, we have that 
 \begin{align}\label{sdlfjsdlkhf}
 	|D_z f(X_j)|= |f(X_j+xg(j-s))-f(X_j)|\leq C\min(1,|xg(j-s)|), 
 \end{align}
 since  $f$ and $f'$ are bounded.
By \eqref{sdlfjsdlkhf} we obtain that 
\begin{align} \label{Dest}
|D_{z} V_n| & \leq  \frac{C}{\sqrt{n}} \sum_{t=1}^n \min(1, |xg(t-s)|)=: A_n(z). 
\end{align}
Furthermore, 
\begin{align}
 {}&   D_{z_1,z_2}^2 V_n =\frac{1}{\sqrt{n}} \sum_{t=1}^n \Big\{ f\Big( x_1g(t-s_1)+x_2g(t-s_2)+X_t\Big)\\ 
 {}& \phantom{D_{z_1,z_2}^2 V_n =\frac{1}{\sqrt{n}} \sum_{t=1}^n \Big\{}  
 -f\Big( x_1g(t-s_1)+X_t\Big)-f\Big( x_2g(t-s_2)+X_t\Big)+f(X_t)\Big\}. 
\end{align}
Again, by  applying the mean-value theorem and using that $f, f'$ and  $f''$ are bounded we obtain the estimate 
\begin{align} \label{D2est}
|D_{z_1, z_2}^2 V_n|  \leq \frac{C}{\sqrt{n}} \sum_{t=1}^n \min(1, |x_1g(t-s_1)|) 
\min(1, |x_2g(t-s_2)|)=: A_n(z_1,z_2). 
\end{align}
Notice that $A_n(z_1,z_2) \leq C\min(A_n(z_1) , A_n(z_2) ).$

 We define the quantity 
\begin{align} \label{rhok}
\rho_k:= \int_{\R} |g(x)g(x+k)|^{\beta/2} \,dx,\qquad k\in\Z.
\end{align}
We will show later that the terms $\gamma_1$ and $\gamma_2$ appearing in Theorems~\ref{theo:SecondOrderPoincare} and \ref{Kdistance}  are both bounded by $Cn^{-1/2}$ solely under the condition $\sum_{k=1}^{\infty} \rho_k < \infty$, cf.\ Lemma~\ref{ljsdlfhslfhg}, while the other terms require the stronger assumption \eqref{gcond}. 

\begin{lem}\label{ljsdlfhslfhg}
	With $\rho_k$ given in \eqref{rhok} we have that $\rho_k\leq C k^{-\alpha\beta/2}$ for all $k\geq 1$. Furthermore, 
	\begin{equation}\label{dslfjdsf}
  \int_\R \Big( \prod_{i=1}^4 |g(t_i-s)|\Big)^{\beta/4}\,ds \leq C |t_2-t_1|^{-\alpha\beta/4} 
  |t_3-t_1|^{-\alpha\beta/4} |t_4-t_1|^{-\alpha\beta/4}, 
\end{equation}
for all $t_1,\dots, t_4\geq 1$, where we use the convention $0^{-r}:=1$ for all $r>0$. 
\end{lem}

\begin{proof}
From the substitution $u=k^{-1} s$ we obtain 
\begin{align}
 \rho_k \leq   {}& \int_0^1 s^{\gamma \beta/2} (k+s)^{-\alpha\beta/2}\,ds+\int_1^\infty s^{-\alpha\beta}
  (k+s)^{-\alpha \beta/2}\,ds \\  =  {}&  k^{1+\gamma\beta/2-\alpha\beta/2} \int_0^{1/k} u^{\gamma \beta/2} (1+u)^{-\alpha\beta/2}\,du+ k^{1-\alpha\beta} \int_{1/k}^\infty u^{-\alpha\beta/2} (1+u)^{-\alpha\beta/2}\,du
  \\   \leq {}& C k^{-\alpha\beta/2}. \label{sldfjsldjf}
\end{align}
Moreover, by the same procedure as in \eqref{sldfjsldjf} and using 
 succesive substitutions we obtain the bound \eqref{dslfjdsf}. 
\end{proof}

\begin{lem}\label{sdfljsdlfj}
The series $v^2$ defined in Theorem~\ref{maintheo} converges absolutely, and  $v_n^2\to v^2$ as $n\to \infty$. 
\end{lem}

\begin{proof}
In the following we will show that 
	\begin{equation}\label{sdfljsdlfjsdfsdfsdf} 
  \sum_{j=1}^\infty |\cov(f(X_j),f(X_0))|<\infty. 
\end{equation}
To prove  \eqref{sdfljsdlfjsdfsdfsdf} we use the covariance identity from Theorem~5.1 in \cite{Last-Penrose-rep} to get 
\begin{equation}\label{dslfjsdhf}
  \textrm{cov}(f(X_j),f(X_0)) = \E\Big[ \int_0^1 \int \E[ D_zf(X_j)\mid \mathcal G_u] \E[ D_z f(X_0)\mid \mathcal G_u] \,\lambda(dz)\,du \Big]
\end{equation}
where $(\mathcal G_u)_{u\in [0,1]}$ are certain $\sigma$-algebras (which will not be important for us). As  noticed  in \cite[Proof of Theorem~1.4]{Last-Penrose-rep} we may always assume that we are in the setting of \cite[Theorem~1.5]{Last-Penrose-rep}. By  Cauchy--Schwarz inequality and the contractive properties of conditional expectation it follows from \eqref{dslfjsdhf} that 
\begin{align}
{}&   \big|\textrm{cov}(f(X_j),f(X_0))\big|\leq 
  \int_0^1 \int \E\Big[\Big|\E[ D_zf(X_j)\mid \mathcal G_u] \E[ D_z f(X_0)\mid \mathcal G_u]\Big|\Big] \,\lambda(dz)\,du
 \\ {}&  \qquad \leq  \int \E[ |D_zf(X_j)|^2]^{1/2}\E[ |D_z f(X_0)|^2]^{1/2} \,\lambda(dz)
 \\ {}& \qquad \leq C  \int_\R \Big( \int_\R  \min(|x^2 g(j-s)g(-s)|,1) \,\nu(dx)\Big) \,ds\\
{}& \qquad \leq C \int_\R |g(j-s)g(-s)|^{\beta/2}\,ds\\
&\qquad=C \rho_j\leq C j^{-\alpha\beta/2} \label{sdfljsdfljh}
\end{align}
 where we have used \eqref{sdlfjsdlkhf} in the third  inequality, the equality follows by the definition of $\rho_j$, and the last inequality follows by Lemma~\ref{ljsdlfhslfhg}.  Since $\alpha\beta>2$, \eqref{sdfljsdfljh} implies \eqref{sdfljsdlfjsdfsdfsdf}.  
 
By \eqref{sdfljsdlfjsdfsdfsdf}, the series $v^2$ converges absolutely. Moreover,  the  stationarity of $(X_j)_{j\in \N}$ implies that 
\begin{align}
  v_n^2={}& \E[V_n^2]= n^{-1} \sum_{j,i=1}^n \cov(f(X_j),f(X_i)) \\ ={}&   \var(f(X_0))+ 2 \sum_{j=1}^n (1-j/n)\cov(f(X_0),f(X_j))\\
 \to   {}& \var(f(X_0))+ 2 \sum_{j=1}^\infty \cov(f(X_0),f(X_j))=v^2 \qquad \quad \textrm{as }n\to\infty, 
\end{align}
where the convergence follows by Lebesgue's dominated convergence theorem together with \eqref{sdfljsdlfjsdfsdfsdf}. 
\end{proof}

\subsection{Bounding the Wasserstein distance}\label{sdljfjsdghfsdg}

Since $v_n^2=\text{var}(V_n) \to v^2$ as $n \to \infty$, cf.\ Lemma~\ref{sdfljsdlfj}, and  $v>0$  by assumption, we note that $v_n$ is bounded away from zero. 
%
Let $\gamma_1,\gamma_2, \gamma_3$ be defined in \eqref{sldfjlsdjfljs}, \eqref{ljsdfghksdhf} and \eqref{sdlfjsdlfh} with 
$F=V_n$. Then,  by  Theorem~\ref{theo:SecondOrderPoincare} and using that $v_n$ is bounded away from 0,  we have that 
\begin{align}\label{dslfjsldfh}
  d_W(V_n/v_n,Z)\leq C(\gamma_1+\gamma_2+\gamma_3).
\end{align}
Using the estimates \eqref{Dest} and \eqref{D2est} we will now compute bounds for the quantities $\gamma_1$, $\gamma_2$ and $\gamma_3$ appearing in the bound for the Wasserstein distance in Theorem~\ref{theo:SecondOrderPoincare}. Notice that the right hand sides in both estimates \eqref{Dest} and \eqref{D2est} are deterministic, so the expectations in the definitions of $\gamma_1, \gamma_2, \gamma_3$ can be omitted. We start with the term $\gamma_1$.

\begin{lem}\label{lem:Gamma1}
There exists a constant $C>0$ such that $\gamma_1\leq Cn^{-1/2}$.
\end{lem}
\begin{proof}
To estimate 
\begin{equation}
  \gamma_1^2 = 4  \int  \E[(D_{z_1} V_n)^2(D_{z_2} V_n)^2]^{1/2} \,
\E[(D_{z_1,z_3}^2 V_n)^2(D_{z_2,z_3}^2 V_n)^2]^{1/2} \lambda^{ 3} (d(z_1,z_2,z_3)),
\end{equation}
 we deduce the following inequality by  \eqref{Dest} and \eqref{D2est}:
\begin{align*}
&\E[(D_{z_1} V_n)^2(D_{z_2} V_n)^2]^{1/2} 
\E[(D_{z_1,z_3}^2 V_n)^2(D_{z_2,z_3}^2 V_n)^2]^{1/2} \\[1.5 ex]
& \qquad \leq \frac{C}{n^2}
\sum_{t_1,\ldots, t_4=1}^n
  \Big\{ \min(1, x_1^2 |g(t_1-s_1) g(t_3-s_1)|) \min(1, x_2^2 |g(t_2-s_2)  g(t_4-s_2)| )
 \\  & \qquad \phantom{\leq \frac{C}{n^2}
\sum_{t_1,\ldots, t_4=1}^n
  \Big\{} \times   \min(1, x_3^2 |g(t_3-s_3) g(t_4-s_3)|  )\Big\}.
\end{align*}
By the substitution $w_i^2=x_i^2y_i$ for $i\in\{1,2,3\}$ we see that for any $y_1, y_2, y_3>0$ it holds that
\[
\int_{\R^3} \min(1, x_1^2 y_1) \min(1, x_2^2 y_2) \min(1, x_3^2 y_3) |x_1x_2x_3|^{-1-\beta} \, dx_1 \,dx_2\, dx_3 
= C y_1^{\beta/2} y_2^{\beta/2} y_3^{\beta/2}.
\]
Indeed, we have that 
\begin{align}
 {}&  \int_{\R^3} \min(1, x_1^2 ) \min(1, x_2^2 ) \min(1, x_3^2 ) |x_1x_2x_3|^{-1-\beta} \,d x_1\,dx_2\,x_3 \\ 
  {}& \qquad =  \Big( \int_{\R} \min(1, x^2 )|x|^{-1-\beta}\,dx\Big)^3<\infty,
\end{align}
since $\beta\in (0,2)$. 
Therefore, we deduce the estimate 
\begin{align}
&   \gamma_1^2 \leq  C \int  \E[(D_{z_1} V_n)^2(D_{z_2} V_n)^2]^{1/2} 
\E[(D_{z_1,z_3}^2 V_n)^2(D_{z_2,z_3}^2 V_n)^2]^{1/2} \lambda^{ 3} (dz_1, dz_2,dz_3) 
\\[1.5 ex]
& \phantom{ \gamma_1^2} \leq \frac{C}{n^2}  \sum_{t_1,\ldots, t_4=1}^n \Big\{
\int_{\R} |g(t_1-s_1) g(t_3-s_1)|^{\beta/2} \,ds_1  \int_{\R} |g(t_2-s_2) g(t_4-s_2)|^{\beta/2} \,ds_2
\\[1.5 ex]
& \phantom{ \gamma_1^2 \leq \frac{C}{n^2}  \sum_{t_1,\ldots, t_4=1}^n \Big\{ }   \times \int_{\R} |g(t_3-s_3) g(t_4-s_3) |^{\beta/2} \,ds_3\Big\}\\
& \phantom{ \gamma_1^2} = \frac{C}{n^2}  \sum_{t_1,\ldots, t_4=1}^n  \rho_{t_1-t_3}\rho_{t_2-t_4}\rho_{t_3-t_4} \leq  \frac{C}{n}  \sum_{v_1,v_2, v_3=-n}^n  \rho_{v_1}\rho_{v_2}\rho_{v_3}\leq  \frac{C}{n} \Big(\sum_{k=0}^\infty \rho_k\Big)^3 \label{sdlfjsdlfj-sdfs}
\end{align} 
where the first equality follows by substitution, the next inequality follows 
by  the change of variables $v_1=t_1-t_3$, $v_2=t_2-t_4$, $v_3=t_3-t_4$, and the last inequality follows from the symmetry 
$\rho_{-k}=\rho_k$. By Lemma~\ref{ljsdlfhslfhg}, we have that   $\sum_{k=0}^\infty \rho_k<\infty$, and hence, \eqref{sdlfjsdlfj-sdfs}  completes the proof of the estimate  
$\gamma_1 \leq Cn^{-1/2}$. 
\end{proof}

Using a similar reasoning, we can also bound the term $\gamma_2$.

\begin{lem}\label{lem:Gamma2}
There exists a constant $C>0$ such that $\gamma_2\leq Cn^{-1/2}$.
\end{lem}
\begin{proof}
Recall that 
\begin{equation}
  \gamma_2^2 = \int 
\E[(D_{z_1,z_3}^2 V_n)^2(D_{z_2,z_3}^2 V_n)^2]\, \lambda^{ 3} (d(z_1,z_2,z_3)).
\end{equation}
By the  inequality \eqref{D2est} we immediately conclude that  
\begin{align*}
{}& \E[(D_{z_1,z_3}^2 V_n)^2(D_{z_2,z_3}^2 V_n)^2] \leq  \frac{C}{n^2}  \sum_{t_1,\ldots, t_4=1}^n \Big\{
 \min(1, x_1^2 |g(t_1-s_1) g(t_2-s_1)|)  \\[1.5 ex] {}& 
\times \min(1, x_2^2 |g(t_3-s_2) g(t_4-s_2)|)
\min(1, x_3^4 |g(t_1-s_3) g(t_2-s_3) g(t_3-s_3) g(t_4-s_3)| )\Big\}.
\end{align*}
As in the proof of Lemma \ref{lem:Gamma1} a substitution shows that for any $y_1, y_2, y_3>0$,
\[
\int_{\R^3} \min(1, x_1^2 y_1) \min(1, x_2^4 y_2) \min(1, x_3^2 y_3) |x_1x_2x_3|^{-1-\beta} \,dx_1\,d x_2\,dx_3
= C y_1^{\beta/2} y_2^{\beta/4} y_3^{\beta/2}.
\]
Therefore, we have the estimate 
\begin{align} \label{est1}
\gamma_2^2={}& \int  \E[(D_{z_1,z_3}^2 F)^2(D_{z_2,z_3}^2 F)^2] \lambda^{ 3} (dz_1, dz_2,dz_3) 
\nonumber \\[1.5 ex]
 \leq {}& \frac{C}{n^2}  \sum_{t_1,\ldots, t_4=1}^n \Big\{
\int_{\R} |g(t_1-s_1) g(t_2-s_1)|^{\beta/2} ds_1  \int_{\R} |g(t_3-s_2) g(t_4-s_2)|^{\beta/2} ds_2
\nonumber \\[1.5 ex]
{}& \phantom{\frac{C}{n^2}  \sum_{t_1,\ldots, t_4=1}^n \Big\{   }\times \int_{\R} |g(t_1-s_3) g(t_2-s_3) g(t_3-s_3) g(t_4-s_3)|^{\beta/4} ds_3\Big\}.
\end{align} 
Now,  the  inequality $|xy|\leq x^2+y^2$, valid for all $x,y\in \R$, implies  
\begin{align*}
&\int_{\R} |g(t_1-s_3) g(t_2-s_3) g(t_3-s_3) g(t_4-s_3)|^{\beta/4} ds_3  \\
&\qquad\leq \int_{\R} |g(t_1-s_3)  g(t_3-s_3) |^{\beta/2} ds_3
+ \int_{\R} | g(t_2-s_3)  g(t_4-s_3)|^{\beta/2} ds_3, 
\end{align*}
which by \eqref{est1}, shows that  
\begin{equation}\label{sdfsdf-sdlfjsdbfg}
  \gamma_2^2
\leq \frac{C}{n} \Big( \sum_{k\geq 0} \rho_k \Big)^3,
\end{equation}
by the same arguments as in the proof of Lemma~\ref{lem:Gamma1}. 
Since $\sum_{k\geq 0} \rho_k<\infty$, cf.\ Lemma~\ref{ljsdlfhslfhg}, 
the estimate $\gamma_2 \leq Cn^{-1/2}$  follows from \eqref{sdfsdf-sdlfjsdbfg}. 
\end{proof}

The final term $\gamma_3$ in the bound for the Wasserstein distance is more subtle. It is this term, which decays slower than $n^{-1/2}$ for certain parameter regimes.

\begin{lem}\label{lem:Gamma3}
There exists a constant $C>0$ such that 
\begin{equation}
  \gamma_3\leq C   \begin{cases}
    n^{-1/2} \qquad \qquad & \textrm{if }\alpha\beta>3, \\ 
    n^{-1/2} \log(n) & \textrm{if }\alpha\beta=3,\\
  	  n^{(2-\alpha\beta)/2}  & \textrm{if } 2<\alpha\beta<3.
  \end{cases}
\end{equation}
\end{lem}
\begin{proof}
Recalling the inequality \eqref{Dest},  
we  have that
\begin{equation}\label{dsfljsdfh}
  \gamma_3 \leq C  \int_{\R^2}  \min \left(|A_n(x,s)|^2, |A_n(x,s)|^3 \right) 
 \lambda (dx,ds).
\end{equation}
From the inequality \eqref{dsfljsdfh}, Lemma~\ref{lem:Gamma3} follows from the result of Lemma~\ref{dsfljsdfhgghd} below. 
\end{proof}

%

%
%
%
%
%
%
%

%
%
%
%
%

\begin{proof}[Proof of Theorem \ref{maintheo} for the Wasserstein distance]
The Wasserstein bound \eqref{wasserstein}  is  a direct consequence of Lemmas~\ref{lem:Gamma1}, \ref{lem:Gamma2} and \ref{lem:Gamma3} and the second-order Poincar\'e inequality \eqref{dslfjsldfh} 
\begin{equation}
 d_W\left( V_n/v_n, Z\right) \leq C(\gamma_1+\gamma_2+\gamma_3)\leq C  \begin{cases}
    n^{-1/2} \qquad \qquad & \textrm{if }\alpha\beta>3, \\ 
    n^{-1/2} \log(n) & \textrm{if }\alpha\beta=3,\\
  	  n^{(2-\alpha\beta)/2}  & \textrm{if }2<\alpha\beta<3.
  \end{cases}
\end{equation}
\end{proof}

The following bound used in the proof of Lemma~\ref{lem:Gamma3} is stated separately as a lemma, 
since we will also use it in the proof of upper bound for the Kolmogorov distance. 

\begin{lem}\label{dsfljsdfhgghd}
Let $p\in [0,2]$ and $q>2$. 
There exists a finite constant $C$ such that 
		\begin{equation}
  \int_{\R^2}  \min \left(|A_n(z)|^p, |A_n(z)|^q \right)  \lambda (dz)\leq C \begin{cases}
    n^{1-q/2} \qquad \qquad & \textrm{if }\alpha\beta>q, \\ 
    n^{1-q/2} \log(n) & \textrm{if }\alpha\beta=q,\\
  	  n^{(2-\alpha\beta)/2}  & \textrm{if }2<\alpha\beta<q.
  \end{cases}
\end{equation}
\end{lem}

\begin{proof}[Proof of Lemma~\ref{dsfljsdfhgghd}]
To obtain the upper bound for the right hand side we need to decompose the integral into different parts according to whether $|x|\in(0,1)$, $|x|\in[1,n^\alpha]$ or $|x|\in(n^\alpha,\infty)$. Using the symmetry in $x$ this means that
\begin{align*}
&\int_{\R^2}  \min \left(|A_n(x,s)|^p, |A_n(x,s)|^q \right) \lambda (ds,dx)\\
&\qquad= 2\left( \int_0^1 \int_{\R} \min \left(|A_n(x,s)|^p, |A_n(x,s)|^q \right)  \lambda (ds,dx)\right. \\
&\qquad \qquad \  + \int_1^{n^{\alpha}} \int_{\R}  \min \left(|A_n(x,s)|^p, |A_n(x,s)|^q \right)  \lambda (ds,dx) \\
& \qquad \qquad \  +  \left.\int_{n^{\alpha}}^{\infty} \int_{\R}  \min \left(|A_n(x,s)|^p, |A_n(x,s)|^q \right)  \lambda (ds,dx)\right) =: I_1 + I_2 +I_3.
\end{align*}
We start by bounding the term $I_1$. For $x >0$ and $s\in [0,n]\setminus \N$ we have that
\begin{align}
\sum_{t=1}^n \min(1, |xg(t-s)|) \leq {}&   
\min(1, |xg(1+[s]-s)|)+\sum_{t=[s]+2}^n |xg(t-s)| 
\\ \leq {}& \min(1,  x (1+[s]-s)^{\gamma})+x \sum_{t=[s]+2}^n (t-s)^{-\alpha} 
\\ \leq {}& \min(1,  x (1+[s]-s)^{\gamma})+ C x=:f_1(s,x)+f_2(s,x),
\end{align}
where we used $g(u)=0$ for all $u<0$ in the first inequality and 
 assumption \eqref{gcond} on $g$ in the second inequality. The third inequality follows from the fact that 
 $\alpha>1$, which is implied by the assumptions $\alpha>2/\beta$ and $\beta<2$.   
For $\gamma<0$ we have 
\begin{align}
 \int_0^n |f_1(s,x)|^q\,ds = {}& n \int_0^1  |\min(1,x s^\gamma) |^q \,ds 
  \\ = {}&  n \Big( x^q \int_{x^{-1/\gamma}}^1  s^{q\gamma} \,ds+ \int_0^{x^{-1/\gamma}} 1 \,ds\Big) 
  \leq C n x^{-1/\gamma},  
\end{align}
where the first equality follows by substitution.
For $\gamma\geq 0$, we have the simple estimate $\int_0^n |f_1(s,x)|^q ds\leq C n x^q$. 
Similarly, we have that $\int_0^n |f_2(s,x)|^q\,ds\leq C nx^q $.   By combining the above estimates  we obtain that 
\begin{align}
  & \int_{0}^{1} x^{-1-\beta} \Big( \int_{0}^{n} |A_n(x,s)|^q\, ds\Big) \, dx \\ 
  &\qquad  \leq Cn^{1-q/2}\int_0^1 x^{-1-\beta} (x^{-1/\gamma}\1_{\{\gamma<0\}} +x^q)\,dx\leq C n^{1-q/2},\label{Psdlhjfsdlkhf}
\end{align}
where the last inequality follows from the assumption $\gamma>-1/\beta$.

For  $s \in (-\infty,0)$ we  use the assumption \eqref{gcond} on $g$ to  obtain
\begin{equation}\label{sdfljsdfh}
  \sum_{t=1}^n \min(1, |xg(t-s)|) \leq C x \sum_{t=1}^n (t-s)^{-\alpha} 
  \leq Cx \Big( (1-s)^{1-\alpha} -(n-s)^{1-\alpha}\Big).
\end{equation}
For $\alpha>1+1/q$ we have that 
\begin{equation}
  \int_{-\infty}^0 |(1-s)^{1-\alpha}-(n-s)^{1-\alpha}|^q\,ds\leq  \int_{-\infty}^0 (1-s)^{q(1-\alpha)}\,ds\leq C, 
\end{equation}
and for $\alpha<1+1/q$ we have that 
\begin{align}
 &  \int_{-\infty}^0 |(1-s)^{1-\alpha}-(n-s)^{1-\alpha}|^q\,ds \leq \int_1^\infty |u^{1-\alpha} -(u+n)^{1-\alpha}|^q\,du\\ 
& \qquad   = n^{q(1-\alpha)+1}\int_{n^{-1}}^\infty | v^{1-\alpha} -(v+1)^{1-\alpha}|^q\,dv 
\\ & \qquad \leq  C n^{q(1-\alpha)+1} \Big( \int_{n^{-1}}^1 v^{q(1-\alpha)}\,dv+ \int_1^\infty v^{-q\alpha}\,dv\Big) 
 \leq C n^{q(1-\alpha)+1},\label{sldfjlsdjflshdd}
\end{align}
where we used $1<\alpha<1+1/q$ in the last inequality.  The above estimates imply for  $\alpha \neq 1+1/q$  that 
\begin{align}\label{sdlfjsdljfs}
 &  \int_{0}^{1} x^{-1-\beta} \Big( \int_{-\infty}^{0}  |A_n(x,s)|^q \,ds\Big) \, dx 
 \\ & \qquad \leq Cn^{-q/2} \int_{0}^{1} x^{q-1-\beta} \,dx 
 \int_{-\infty}^0 |(1-s)^{1-\alpha}-(n-s)^{1-\alpha}|^q\,ds \\ & \qquad  \leq C(n^{-q/2}+n^{1-q\alpha+q/2})\leq C n^{1-q/2},\label{sdlfjsdlj}
\end{align}
where the last inequality follows since $\alpha>1$. For $\alpha=1+1/q$, assumption \eqref{gcond} is satisfied for $\tilde \alpha=\alpha-\epsilon$ for all $\epsilon>0$ small enough.   
Hence, by \eqref{sldfjlsdjflshdd} used with $\tilde \alpha$ we obtain that \eqref{sdlfjsdljfs} is  bounded by $C n^{1-q/2}$ by choosing $\epsilon$ small enough.  

The assumption that $g(x)=0$ for all $x<0$, implies that  $A_n(x,s)=0$ for all $s>n$, and hence \eqref{Psdlhjfsdlkhf} and \eqref{sdlfjsdlj}   show that 
%
\begin{align} \label{i1}
I_1 \leq C n^{1-q/2}.
\end{align} 
Next, we treat the term $I_3$.  For the integral
$$
 \int_{n^{\alpha}}^{\infty}  \int_{-\infty}^{-n} \min \left(|A_n(x,s)|^p, |A_n(x,s)|^q \right)  \lambda (ds, dx)
$$
 we need to distinguish different cases, namely  $s\leq -x^{1/\alpha}$,   $-x^{1/\alpha}<s\leq n-x^{1/\alpha}-1$ and $n-x^{1/\alpha}-1<s\leq n$. 
 
We start with the case $s\leq -x^{1/\alpha}$. Note that 
\begin{align} \label{Anxest}
|A_n(x,s)| \leq xn^{-1/2} \sum_{t=1}^n |g(t-s)| \leq C x\sqrt{n} (-s)^{-\alpha} 
\end{align}
and observe that $x\sqrt{n} (-s)^{-\alpha} >1$ if and only if $s>-x^{1/\alpha}n^{1/(2\alpha)}$. 
We obtain the inequality
\begin{align}
{}&   \int^{-x^{1/\alpha}}_{-x^{1/\alpha}n^{1/(2\alpha)}} (-s)^{-\alpha p}\,ds \leq Cx^{(1-\alpha p)/\alpha}\left(1+n^{(1-\alpha p)/(2\alpha)} \right). 
\end{align}
On the other hand, we have that 
\begin{align}
{}&   \int_{-\infty}^{-x^{1/\alpha}n^{1/(2\alpha)}} (-s)^{-\alpha q}\,ds \leq C x^{(1-\alpha q)/\alpha}n^{(1-\alpha q)/(2\alpha)}. 
\end{align}
By \eqref{Anxest} we thus conclude that
\begin{align}
{}&   \int_{n^{\alpha}}^{\infty}  \int_{-\infty}^{-x^{1/\alpha}} \min \left(|A_n(x,s)|^p, |A_n(x,s)|^q \right)  \lambda (ds, dx)\\ 
{}& \leq C \int_{n^{\alpha}}^{\infty} x^{-1-\beta} \left(
 x^pn^{p/2}\int^{-x^{1/\alpha}}_{-x^{1/\alpha}n^{1/(2\alpha)}} (-s)^{-\alpha p}\,ds +  x^qn^{q/2} \int_{-\infty}^{-x^{1/\alpha}n^{1/(2\alpha)}} (-s)^{-\alpha q}\,ds
 \right) \,dx \\
{}& \leq C\left(n^{1-\alpha \beta +1/(2\alpha)} + n^{1-\alpha \beta +p/2} \right). 
\end{align}
%
For $x>1$ and  $-x^{1/\alpha}<s\leq n-x^{1/\alpha}-1$ we obtain
\begin{align}
 {}&  \sum_{t=1}^n \min(1, |xg(t-s)|) = \sum_{t=1}^{[s+x^{1/\alpha}]} 1 + 
\sum_{t=[s+x^{1/\alpha}]+1}^n x(t-s)^{-\alpha}
 \\{}&  \qquad \leq C\Big(   (s+x^{1/\alpha}) +  x\Big( (x^{1/\alpha})^{1-\alpha}    -(n-s)^{1-\alpha}\Big) \Big). \label{sdlfjsdsdfsdf}
\end{align}
The substitution $v=x^{-1/\alpha}(n-s)$ yields
\begin{align}
{}&   \int_{-x^{1/\alpha}}^{n-x^{1/\alpha}-1} \Big|x\Big(  (x^{1/\alpha})^{1-\alpha}    -(n-s)^{1-\alpha}\Big)\Big|^p\,ds 
  \\  {}& \qquad = 
     x^{(p+1)/\alpha}\int_{1+x^{-1/\alpha}}^{1+nx^{-1/\alpha}} | 1    -v^{1-\alpha}|^2\,dv
  \leq C n x^{p/\alpha}, \label{sldfjlsdjhhs}
\intertext{and} 
{}&   \int_{-x^{1/\alpha}}^{n-x^{1/\alpha}-1} | s+x^{1/\alpha}|^p\,ds \leq \int_0^n u^p\,du = \frac{1}{p+1} n^{p+1}. \label{sdlfhjsdgf}
\end{align}
From \eqref{sdlfjsdsdfsdf}, \eqref{sldfjlsdjhhs} and \eqref{sdlfhjsdgf} we  deduce that
\begin{align}
 {}&  \int_{n^{\alpha}}^{\infty} x^{-1-\beta} \Big(\int_{-x^{1/\alpha}}^{n-x^{1/\alpha}-1}   |A_n(x,s)|^p \,ds\Big)\, dx 
  \\  {}& \qquad  \leq C n^{-p/2}\Big( n^{p+1} \int_{n^\alpha}^\infty x^{-1-\beta}\,dx+n\int_{n^\alpha}^\infty x^{-1-\beta+p/\alpha}\,dx\Big) 
  \\ {}& \qquad \leq  C n^{1-\alpha \beta +p/2}, 
\end{align}
where we used the assumption $\alpha \beta>2$ in the second  inequality.

Finally, for the last case $n-x^{1/\alpha}-1<s\leq n$ we have
\[
\sum_{t=1}^n \min(1, |xg(t-s)|) \leq  n,
\]
which leads to
\begin{equation}
\int_{n^{\alpha}}^{\infty} x^{-1-\beta} \Big(\int_{n-x^{1/\alpha}-1}^{n}   |A_n(x,s)|^p \,ds\Big)\, dx 
 \leq C n^{p/2} \int_{n^{\alpha}}^{\infty} x^{-1-\beta+1/\alpha} \, dx \leq Cn^{1-\alpha \beta +p/2}. 
 \end{equation}
Summarizing, we arrive at the bound
\begin{align} \label{i3}
I_3 \leq C\left(n^{1-\alpha \beta +1/(2\alpha)} + n^{1-\alpha \beta +p/2} \right).
\end{align}

Next, we will bound the term  $I_2$   as follows: 
\begin{align}
  I_2\leq C\Big\{ {}&  \int_1^{n^{\alpha}} x^{-1-\beta}\Big(\int_{-n}^n   \min \left(|A_n(x,s)|^p, |A_n(x,s)|^q \right) \,ds\Big)\,dx 
 \\ {}& +\int_1^{n^{\alpha}} x^{-1-\beta}\Big(\int_{-\infty}^{-n}   \min \left(|A_n(x,s)|^p, |A_n(x,s)|^q \right) \,ds\Big)\,dx \Big\}=J_1+J_2, 
\end{align}
where we recall that $A_n(x,s)=0$ for $s>n$. 
To estimate $J_1$ we have for  $s\in \R$ and  $x>1$ that 
 \begin{align}
 {}&  \sum_{t=1}^n \min(1, |xg(t-s)|) \leq  \sum_{t=[s]+1}^{[s+x^{1/\alpha}]} 1 + 
\sum_{t=[s+x^{1/\alpha}]+1}^n x(t-s)^{-\alpha}
 \\{}&  \qquad \leq C
 \begin{cases}
 	x^{1/\alpha} +x\Big( (x^{1/\alpha})^{1-\alpha}    -(n-s)^{1-\alpha}\Big) \qquad \qquad &  \textrm{if }s+x^{1/\alpha}\leq n, \\
 n-s    & \textrm{if }s+x^{1/\alpha}> n,
 \end{cases}
 \\{}&  \qquad \leq C x^{1/\alpha}.\label{lksjdflhsdklfh}
 \end{align}
We note  that $x^{1/\alpha}n^{-1/2}\leq 1$ if and only if $x\leq n^{\alpha/2}$, and write   $J_1$ as $J_1=J_1'+J_1''$. Note that
\begin{align}
 J_1':={}& \int_1^{n^{\alpha/2}} x^{-1-\beta}\Big(\int_{-n}^n   \min \left(|A_n(x,s)|^p, |A_n(x,s)|^q \right) \,ds\Big)\,dx 
 \\ \leq {}& C
\int_1^{n^{\alpha/2}} \Big(\int_{-n}^{n}   |x^{1/\alpha}n^{-1/2}|^q\,ds \Big)x^{-1-\beta}\,dx
 \leq n^{1-q/2}\int_1^{n^{\alpha/2}}   x^{-1-\beta+q/\alpha} \,dx
 \\  \leq C {}&\begin{cases}
    n^{1-q/2} \qquad \qquad & \textrm{if }\alpha\beta>q \\ 
    n^{1-q/2} \log(n) & \textrm{if }\alpha\beta=q\\
  	  n^{(2-\alpha\beta)/2}  & \textrm{if }2<\alpha\beta<q, 
  \end{cases}
 \label{dsfljsdlhf}
\end{align}
where we have used \eqref{lksjdflhsdklfh} in the first inequality. 
Furthermore, 
\begin{align}
  J_1'':={}&\int_{n^{\alpha/2}}^{n^{\alpha}} x^{-1-\beta}\Big(\int_{-n}^n   \min \left(|A_n(x,s)|^p, |A_n(x,s)|^q \right) \,ds\Big)\,dx 
\\  \leq {}& C
\int_{n^{\alpha/2}}^{n^\alpha} \Big(\int_{-n}^{n}   |x^{1/\alpha}n^{-1/2}|^p\,ds \Big)x^{-1-\beta}\,dx
 \leq C n^{1-p/2} \int_{n^{\alpha/2}}^{n^{\alpha}}   x^{-1-\beta+p/\alpha} \,dx\\
  = {}& C n^{1-\alpha \beta +p/2} \int_{n^{-1/2}}^1 v^{p-1-\alpha\beta}\,dv  
  \leq  C n^{(2-\alpha\beta)/2},\label{lsjdfljsdhf}
\end{align}
where we have applied \eqref{lksjdflhsdklfh} in the first inequality, and the substitution  $v=n^{-1} x^{1/\alpha}$ in the second equality. 

To  estimate $J_2$, again we need to distinguish several cases. We recall the inequality \eqref{Anxest} and the statement below it, and notice that $-x^{1/\alpha}n^{1/(2\alpha)}>-n$ if and only if $x<n^{\alpha -1/2}$. We obtain the estimate 
\begin{align}
{}&   \int_1^{n^{\alpha-1/2}} x^{-1-\beta}\Big(\int_{-\infty}^{-n}   \min \left(|A_n(x,s)|^p, |A_n(x,s)|^q \right) \,ds\Big)\,dx  \\ 
{}&  \qquad  \leq C n^{1-q\alpha +q/2} \int_1^{n^{\alpha-1/2}} x^{-1-\beta+q} \,dx \leq C n^{1-\alpha \beta +\beta/2}. 
\end{align}
Recalling again the inequality \eqref{Anxest} we deduce that
\begin{align}
{}&   \int_{n^{\alpha-1/2}}^{n^\alpha} x^{-1-\beta}\Big(\int_{-\infty}^{-x^{1/\alpha} n^{1/(2\alpha)}}   
\min \left(|A_n(x,s)|^p, |A_n(x,s)|^q \right) \,ds\Big)\,dx  \\ 
{}&  \qquad  \leq C n^{1/(2\alpha)} \int_{n^{\alpha-1/2}}^{n^\alpha} x^{-1-\beta+1/\alpha} \,dx \leq C n^{1-\alpha \beta +\beta/2}. 
\end{align}
Finally, we also get for $p\not = \beta$ that 
\begin{align} \label{pbeta}
{}&   \int_{n^{\alpha-1/2}}^{n^\alpha} x^{-1-\beta}\Big(\int_{-x^{1/\alpha} n^{1/(2\alpha)}}^{-n}    
\min \left(|A_n(x,s)|^p, |A_n(x,s)|^q \right) \,ds\Big)\,dx  \\ 
{}&  \qquad  \leq C n^{1-\alpha p +p/2} \int_{n^{\alpha-1/2}}^{n^\alpha} x^{-1-\beta+p} \,dx \leq C 
\left( n^{1-\alpha \beta +\beta/2} + n^{1-\alpha \beta +p/2} \right). 
\end{align}
Next, we summarise our findings. Since $\alpha>1$ we have that $1-\alpha\beta + \beta/2 < (2-\alpha\beta)/2$. On the other hand,
$\alpha \beta>2\geq p$ implies the inequality $1-\alpha\beta + p/2 < (2-\alpha\beta)/2$ (when $p=\beta$ 
an additional $\log n$ factor appears in \eqref{pbeta}, but both rates are still dominated by $n^{(2-\alpha \beta)/2}$). 
Thus, we conclude from \eqref{dsfljsdlhf} and
\eqref{lsjdfljsdhf}  that 
\begin{align}
 I_2 \leq C {}&\begin{cases}
    n^{1-q/2} \qquad \qquad & \textrm{if }\alpha\beta>q, \\ 
    n^{1-q/2} \log(n) & \textrm{if }\alpha\beta=q,\\
  	  n^{(2-\alpha\beta)/2}  & \textrm{if }2<\alpha\beta<q.
  \end{cases}
 \label{I2esti}
\end{align}
Due to \eqref{i1} and  \eqref{i3} we obtain the desired assertion since $1-\alpha\beta + 1/(2\alpha) < (2-\alpha\beta)/2$. 
\end{proof}

\subsection{Bounding the Kolmogorov distance}

We let $\gamma_1, \gamma_2,\overline \gamma_3, \gamma_4, \gamma_5$
 and $\gamma_6$ be as defined in \eqref{sldfjlsdjfljs}, \eqref{ljsdfghksdhf}  and \eqref{sldfjsldjhf}--\eqref{sldfjsdgf} with $F=V_n$. 
 By our  second-order Poincar\'e inequality  Theorem~\ref{Kdistance}, and using the fact that $v_n$ is bounded away from zero we have  that 
\begin{equation}\label{sdlfjsdkgfsd}
 d_K\left( V_n/v_n, Z\right)\leq C (\gamma_1+\gamma_2+\overline{\gamma}_3 + \gamma_4+\gamma_5+\gamma_6).
\end{equation}
In this subsection we will  bound the terms $\overline{\gamma}_3$, $\gamma_4$, $\gamma_5$ and $\gamma_6$ to obtain the 
bound for the Kolmogorov distance in Theorem~\ref{maintheo}. Throughout the proof we consider a continuously differentiable function  $\varphi:\R\to\R$ with bounded derivative, whose support $\text{supp}(\varphi)$ is contained in the interval $[-2,2]$, which satisfies $\varphi(x)=1$ for $x\in [-1,1]$ and is such that $\| \varphi \|_{\infty}=1$. In particular, this ensures that $\| \varphi' \|_{\infty}<\infty$.  We start with the term $\overline{\gamma}_3$, which we  handle  as $\gamma_3$ in Lemma~\ref{lem:Gamma3}.

\begin{lem}\label{lem:Gamma34}
There exists a constant $C>0$ such that 
\begin{equation}\label{sdlfjsdlfj}
  \overline{\gamma}_3 \leq C  \begin{cases}
    n^{-1} \qquad \qquad & \textrm{if }\alpha\beta>4 \\ 
    n^{-1} \log(n) & \textrm{if }\alpha\beta=4\\
  	  n^{(2-\alpha\beta)/2}  & \textrm{if }2<\alpha\beta<4.
  \end{cases}
  \end{equation}
\end{lem}
\begin{proof}
Applying  the inequality \eqref{Dest} we conclude that
\begin{align}
  \overline{\gamma}_3 \leq  {}& C \int_{\R^2} \E[ (D_z V_n)^2\1_{\{|D_z V_n|>1\}} ]^{1/2} \E[(D_z V_n)^2]^{1/2}\,\lambda(dz) 
  \\  \leq {}& C \int_{\R^2} A_n(x,s)^2 1_{\{|A_n(x,s)|>1\}}\, \lambda(dx,ds) \\  \leq {}& C \int_{\R^2}  \min \left(A_n(x,s)^2, A_n(x,s)^4 \right) 
 \lambda (dx,ds),
\end{align}
which together with Lemma~\ref{dsfljsdfhgghd} implies \eqref{sdlfjsdlfj}. 
%
\end{proof}

\begin{lem}\label{sdlfjsdlfjhgsa}
There exists a finite constant $C$ such that 
		\begin{align}\label{sfsdfhs-1}
	{}& \int A_n(z)^2\,\lambda(dz) \leq C, 
\intertext{and}
\label{sfsdfhs-2}
  {}& 
	\int A_n(z)^4\,\lambda(dz)\leq C
\begin{cases}
 	n^{-1} & \text{if }  \alpha\beta >4, \\ 
 	n^{-1}(\log(n))^3  & \text{if } \alpha\beta=4,\\
 	 	n^{2-\frac{3}{4}\alpha\beta}  & \text{if } 2<\alpha\beta<4. 
 	 	 \end{cases}
\end{align}
\end{lem}

\begin{proof}[Proof of Lemma~\ref{sdlfjsdlfjhgsa}] 
To show \eqref{sfsdfhs-1} we proceed as follows:
\begin{align}
  {}&  \int A_n(z)^2\,\lambda(dz) \\
{}&  \qquad = \frac{1}{n} \sum_{t_1,t_2=1}^n \int \Big( \int \min(1,|x g(t_1-s)|)\min(1,|xg(t_2-s)|) |x|^{-1-\beta}\,dx\Big) \,ds\\ 
{}& \qquad \leq  \frac{1}{n} \sum_{t_1,t_2,=1}^n \int \Big( \int \min(1,|x^2 g(t_1-s)g(t_2-s)|) |x|^{-1-\beta}\,dx\Big) \,ds
\\ {}& \qquad \leq \frac{C}{n} \sum_{t_1,t_2,=1}^n \int |g(t_1-s)g(t_2-s)|^{\beta/2} \,ds =\frac{C}{n} \sum_{t_1,t_2,=1}^n \rho_{t_1-t_2} \\ {}& \qquad \leq C \sum_{t=0}^n \rho_t \leq C \sum_{t=1}^n t^{-\alpha\beta/2} \leq C, \label{sdlfjsldfj}
\end{align}
where the second inequality follows from the substitution $u=x^2 g(t_1-s) g(t_2-s)$, and the last inequality is a consequence of the assumption that $\alpha\beta>2$.

Our proof of \eqref{sfsdfhs-2} relies on the estimate
		\begin{equation}
	\int A_n(z)^4\,\lambda(dz)\leq  C 
  \frac{1}{n^2} \sum_{t_1,\dots, t_4=1}^n \int_\R \Big( \int_\R \prod_{i=1}^4 \min\{1,|xg(t_i-s)|\}\, |x|^{-1-\beta}\,dx\Big) \,ds. \label{sldjflsdjfl}
\end{equation}
Moreover, we have
\begin{align}
  {}& \int_\R \prod_{i=1}^4 \min\{1,|xg(t_i-s)|\}\, |x|^{-1-\beta}\,dx\leq \int_\R \min\Big\{1,x^4 \prod_{i=1}^4 |g(t_i-s)|\Big\}|x|^{-1-\beta}\,dx \\ {}& \qquad = \frac{1}{4}
  \Big( \prod_{i=1}^4 |g(t_i-s)|\Big)^{\beta/4} \int_\R \min\{1,|u|\} |u|^{-1-\beta/4}\,du\leq C  \Big( \prod_{i=1}^4 |g(t_i-s)|\Big)^{\beta/4}, \qquad \label{sljdflsjdfl}
\end{align}
where the equality follows by the substitution $u=x^4 \prod_{i=1}^4 |g(t_i-s)|$. 
From the two estimates \eqref{sldjflsdjfl} and \eqref{sljdflsjdfl} we obtain  
\begin{align}
{}&  	\int A_n(z)^4\,\lambda(dz)\leq C \frac{1}{n^2} \sum_{t_1,\dots, t_4=1}^n \int_\R \Big( \prod_{i=1}^4 |g(t_i-s)|\Big)^{\beta/4}\,ds \\
{}& \qquad  \leq \frac{C}{n}\Big(\sum_{r=1}^n r^{-\alpha\beta/4}\Big)^3 \leq C 
\begin{cases}
 	n^{-1} & \text{if }  \alpha\beta >4, \\ 
 	n^{-1}(\log(n))^3  & \text{if } \alpha\beta=4,\\
 	 	n^{2-\frac{3}{4}\alpha\beta}  & \text{if } 2<\alpha\beta<4, 
 \end{cases}
\end{align}
where the second inequality follows by \eqref{dslfjdsf}. This completes the proof of the lemma. 
\end{proof}

\begin{lem}\label{ljsdlfjlsdh}
There exists a constant $C$ such that 
\begin{equation}\label{ljsdlfhsdf}
  \gamma_4\leq C 
  \begin{cases}
     n^{-1}  & \text{if } \alpha\beta>4, \\
     n^{-1} \log(n)\qquad   & \text{if } \alpha\beta=4,  \\
     n^{(2-\alpha\beta)/2}  & \text{if } 8/3\leq \alpha\beta<4,\\
  	n^{3- \frac{5}{4} \alpha\beta} & \text{if } 2<\alpha\beta<8/3.
  	
  \end{cases}
\end{equation}
\end{lem}

\begin{proof}
By our choice of the function $\varphi$ we have that 
	\begin{equation}\label{sdfsdf}
 \gamma_4\leq C\left(  (\E V_n^4)^{1/4} + 1\right)
\int \E[(D_zV_n)^4  \1_{\{|D_z V_n|\leq 1\}}]^{1/2}  \E[ |D_z V_n|^4]^{1/4}  \, \lambda(dz).
\end{equation}
The  inequality $x^4 1_{\{|x| \leq 1 \}} \leq \min (x^4,x^2)$ implies that 
\begin{align}
 {}&  \int \E[|D_zV_n|^4  \1_{\{|D_z V_n|\leq 1\}}]^{1/2}  \E[ |D_z V_n|^4]^{1/4}  \, \lambda(dz)
 \\ 
  {}& \qquad \leq \int \E[\min(|D_z V_n|^2, |D_zV_n|^4)]^{1/2}  \E[ |D_z V_n|^4]^{1/4}  \, \lambda(dz)
 \\ {}&  \qquad \leq \int \min(A_n(z)^2,A_n(z)^4)  \, \lambda(dz) \leq C \begin{cases}
    n^{-1} \qquad \qquad & \textrm{if }\alpha\beta>4, \\ 
    n^{-1} \log(n) & \textrm{if }\alpha\beta=4,\\
  	  n^{(2-\alpha\beta)/2}  & \textrm{if }2<\alpha\beta<4,
  \end{cases}
  \label{dsfljsdlfj}
\end{align}
where the last inequality follows by Lemma~\ref{dsfljsdfhgghd}.

 Lemma~4.2 of \cite{LPS} shows  that 
	\begin{equation}\label{sdlfjsdf}
  \E[V_n^4]\leq C \max\Big\{\int (\E[ (D_z V_n)^4])^{1/2}\lambda(dz), \int \E[(D_z V_n)^4]\,\lambda(dz), 1\Big\}. 
\end{equation}
Hence, a combination of \eqref{sdlfjsdf}, the inequality $|D_z V_n|\leq A_n(z)$, cf.\ \eqref{Dest}, and \eqref{sfsdfhs-1}--\eqref{sfsdfhs-2} of Lemma~\ref{sdlfjsdlfjhgsa} implies that
\begin{equation}\label{sdfljsdlfj-sdfs}
  (\E[ V_n^4])^{1/4}\leq C 
  \begin{cases}
   1 \qquad  &  \text{if }\alpha\beta \geq 8/3, \\
   n^{(2-\frac{3}{4}\alpha\beta)/4}\qquad  & \text{if } 2<\alpha\beta <8/3. 	
  \end{cases}
\end{equation}
 The two inequalities \eqref{dsfljsdlfj}  and \eqref{sdfljsdlfj-sdfs} 
yield  the bound \eqref{ljsdlfhsdf}, which completes the proof of the lemma. 
 \end{proof}

In the next step we treat the term $\gamma_5$.

\begin{lem}\label{lem:Gamma5}
There exists a constant $C>0$ such that 
\begin{equation}\label{sdlfjsdhhs}
  \gamma_5\leq C  \begin{cases}
    n^{-1/2} \qquad \qquad & \textrm{if }\alpha\beta>4 \\ 
    n^{-1/2} \log(n)^{1/2} & \textrm{if }\alpha\beta=4\\
  	  n^{(2-\alpha\beta)/4}  & \textrm{if }2<\alpha\beta<4
  \end{cases}
\end{equation}
\end{lem}
\begin{proof}
We use the inequality $x^4 1_{\{|x| \leq 1 \}} \leq \min (x^4,1)$ to obtain the upper bound
\begin{align}
	\gamma_5^2 \leq {}& \int \E[ (D_z V_n)^4\1_{\{|D_z V_n|\leq 1\}}]^{1/2} \E[(D_zV_n)^4]^{1/2} \, \lambda(dz)\\  \leq {}& \int \min\left(A_n(z)^4, A_n(z)^2 \right) \lambda(dz). \label{lsjdfljsd}
\end{align}
Lemma~\ref{dsfljsdfhgghd} completes the proof. 
\end{proof}

Finally, we derive an upper bound for the term $\gamma_6$.

\begin{lem}\label{lem:Gamma6}
There exists a constant $C>0$ such that 
\begin{equation} \label{slfjsdlfj}
\gamma_{6} \leq C
\begin{cases}
n^{-1/2} & \text{if } \alpha\beta>4, \\
n^{-1/2}\log(n)\qquad  & \text{if } \alpha\beta=4, \\
n^{(2-  \alpha\beta)/4}  & \text{if } \frac{8}{3}\leq \alpha\beta<4,\\
n^{(3- \frac{5}{4} \alpha\beta)/2}  & \text{if } 2<\alpha\beta<\frac{8}{3}.
\end{cases}
\end{equation}
\end{lem}

\begin{proof}
Since the quantity $\gamma_6^2$ consists of three terms, we use the decomposition
\[
\gamma_6^2 =: 3\left(\gamma_{6.1}^2 + \gamma_{6.2}^2 +\gamma_{6.3}^2\right),
\]  
with the terms $\gamma_{6.1}$, $\gamma_{6.2}$ and $\gamma_{6.3}$ given by
\begin{align*}
\gamma_{6.1}^2 &:= \int   \E[ (D_{z_2}(\varphi(D_{z_1}V_n) D_{z_1}V_n))^4]^{1/2} \E[|D_{z_1} V_n|^4]^{1/2}\,\lambda^{ 2}(d(z_1,z_2)),\\
\gamma_{6.2}^2 &:= \int \E[\varphi(D_{z_1}V_n)^2 (D_{z_1}V_n)^4 ]^{1/2} \E[|D^2_{z_2,z_1} V_n|^4]^{1/2}\,\lambda^{ 2}(d(z_1,z_2)),\\
\gamma_{6.3}^2 &:= \int \E[ (D_{z_2}(\varphi(D_{z_1}V_n) D_{z_1}V_n))^4]^{1/2} \E[|D^2_{z_2,z_1} V_n|^4]^{1/2} \,\lambda^{ 2}(d(z_1,z_2)).
\end{align*}
We first prove the inequality 
\begin{equation} \label{neben}
\begin{split}
M &:=\int   \E[(D_{z_1}V_n)^4 ]^{1/2} \E[(D^2_{z_2,z_1} 
V_n)^4]^{1/2} \,\lambda^{ 2}(d(z_1,z_2)) \\
&\leq C
\begin{cases}
n^{-1} & \text{if } \alpha\beta>4, \\
n^{-1} (\log(n))^2 \qquad & \text{if } \alpha\beta=4, \\
n^{(2-\alpha\beta)/2} & \text{if } 2<\alpha\beta <4.
\end{cases}
\end{split}
\end{equation}
Indeed, by applying the estimates \eqref{Dest} and \eqref{D2est}, we conclude that 
\begin{align} \label{Mest}
M\leq \int   A_n(z_1)^2 A_n(z_1, z_2)^2\,  \lambda^{ 2}(d(z_1,z_2)).
\end{align}
Following the same arguments as in the proofs of Lemma \ref{lem:Gamma1} and Lemma \ref{lem:Gamma2}, we deduce the inequality
\begin{align*} 
  M\leq {}&\frac{C}{n^2} \sum_{t_1,\dots,t_4=1}^n \Big[\int \Big(\int\min\Big(1, x^4\prod_{i=1}^4 |g(t_i-s)|\Big) |x|^{-1-\beta}\,dx\Big)\,ds \\ {}& 
\phantom{ \frac{C}{n^2} \sum_{t_1,\dots,t_4=1}^n \Big[}
 \times \int \Big(  \int \min\Big(1,x^2 |g(t_3-s)g(t_4-s)|\Big)|x|^{-1-\beta}\,dx\Big)\,ds\Big]
 \\   \leq {}& \frac{C}{n^2}  \sum_{t_1,\ldots, t_4=1}^n \Big[ \int_{\R} |g(t_1-s) g(t_2-s) g(t_3-s) g(t_4-s)|^{\beta/4} \,ds  \\[1.5 ex]
{}& \phantom{\frac{C}{n^2}  \sum_{t_1,\ldots, t_4=1}^n \Big[}
\times\int_{\R} |g(t_3-s) g(t_4-s)|^{\beta/2}\, ds\Big].
\end{align*} 
Hence, by Lemma~\ref{ljsdlfhslfhg} we have that 
\begin{align}
	M\leq \frac{C}{n} \Big(\sum_{r=1}^n r^{-\alpha\beta/4}\Big)^2 \Big(\sum_{r=1}^n r^{-\frac{3}{4}\alpha\beta}\Big) 
	\leq C 
	\begin{cases}
		n^{-1} & \text{if } \alpha\beta>4,\\
		n^{-1}(\log(n))^2 & \text{if }\alpha\beta =4, \\ 
		n^{1-\alpha\beta/2}   & \text{if }2<\alpha\beta <4, 
			\end{cases}
\end{align}
which shows  \eqref{neben}. 

Now, we start deriving the bounds for the terms $\gamma_{6.1}^2$, $\gamma_{6.2}^2$ and $\gamma_{6.3}^2$. 
First, we consider the quantity  $\gamma_{6.1}^2$. 
 We have that
\begin{align}
	{}& |D_{z_2}(\varphi(D_{z_1}V_n))|= 
\left| \varphi(D_{z_1}V_n + D^2_{z_1, z_2}V_n) - \varphi(D_{z_1}V_n ) \right|\\
{}& \qquad \leq \1_{\{|D_{z_1}V_n|\leq 3\}} C | D^2_{z_1, z_2}V_n|+ \1_{\{|D_{z_1}V_n| > 3\}}C  \1_{\{|D_{z_1,z_2}^2V_n|> 1\}}, 
\label{sljsdfhsdlfhggsg}
\end{align}
where the inequality follows by using the mean-value theorem and the fact that $\varphi'$ is bounded for the first term, and for the second term we use that  $\varphi$ is bounded and  has support in $[-2,2]$. 

By \eqref{sljsdfhsdlfhggsg} we obtain the decomposition
\begin{align}
 \gamma_{6.1}^2 \leq {}&  C \Big( 
\int  \E[ (\1_{\{|D_{z_1}V_n|\leq 3\}}  | D^2_{z_1, z_2}V_n|D_{z_1}V_n)^4]^{1/2} \E[|D_{z_1} V_n|^4]^{1/2} \, \lambda^{ 2}(d(z_1,z_2)) \\ 
{}& \phantom{C \Big( }+ \int  \E[ (  \1_{\{|D_{z_1}V_n| > 3\}} \1_{\{|D_{z_1,z_2}^2V_n|> 1\}}D_{z_1}V_n)^4]^{1/2} \E[|D_{z_1} V_n|^4]^{1/2} \, \lambda^{ 2}(d(z_1,z_2)) \Big) 
\\ \leq {}& C \Big( M + \int |A_n(z)|^4\,\lambda(dz)\times \int \1_{\{|A_n(z)|>1\}}\,\lambda(dz)\Big), 
\label{sdflhsgdfgsdf}
%
%
\end{align}
where  the second inequality follows by the two estimates $|D^2_{z_1,z_2} V_n|\leq  A_n(z_2)$, cf.\  \eqref{D2est} and the line following it, and $|D_{z_1} V_n| \leq  A_n(z_1)$, cf.\  \eqref{Dest}.  Lemma~\ref{sdlfjsdlfjhgsa} shows  that 
\begin{equation}\label{sljdfsldfj}
  \int  A_n(z)^4\, \lambda(dz) \leq  C 
\begin{cases}
 	n^{-1} & \text{if }  \alpha\beta >4, \\ 
 	n^{-1}(\log(n))^3  & \text{if } \alpha\beta=4, \\
 	 n^{2-\frac{3}{4}\alpha\beta}  & \text{if } 2<\alpha\beta<4, 
 \end{cases}
 \end{equation}
 and  Lemma~\ref{dsfljsdfhgghd} implies that 
 \begin{align}\label{sdlfjsdhgfss}
  \int     1_{\{A_n(z)>1\}}\, \lambda(dz)\leq   {}& 
\int_{\R^2}  \min \left(1, A_n(x,s)^4 \right) 
 \lambda (dx,ds)
\\  
 \leq {}& C  \begin{cases}
    n^{-1} \qquad \qquad & \textrm{if }\alpha\beta>4, \\ 
    n^{-1} \log(n) & \textrm{if }\alpha\beta=4,\\
  	  n^{(2-\alpha\beta)/2}  & \textrm{if }2<\alpha\beta<4.
  \end{cases}
\end{align}
Hence from \eqref{sdflhsgdfgsdf}, \eqref{sljdfsldfj}, \eqref{sdlfjsdhgfss} and \eqref{neben}  we deduce the inequality 
\begin{align} \label{gamma6.1}
\gamma_{6.1}^2 \leq C
\begin{cases}
n^{-1} & \text{if } \alpha\beta>4, \\
n^{-1} (\log(n))^2\qquad  & \text{if } \alpha\beta=4, \\
n^{3-\frac{5}{4}\alpha\beta} & \text{if } 2<\alpha\beta<4.
\end{cases}
\end{align}

Since $\varphi$ is bounded we have that 
\begin{align}
  \label{sdfljsdgs}
\gamma_{6.2}^2 = {}& \int  \E[\varphi(D_{z_1}V_n)^2 (D_{z_1}V_n)^4 ]^{1/2} \E[|D^2_{z_2,z_1} V_n|^4]^{1/2} \, \lambda^{ 2}(d(z_1,z_2))\leq C M,
\intertext{and}
\gamma_{6.3}^2 ={}&  
\int  \E[ (D_{z_2}(\varphi(D_{z_1}V_n) D_{z_1}V_n))^4]^{1/2} \E[|D^2_{z_2,z_1} V_n|^4]^{1/2}\, \lambda^{ 2}(d(z_1,z_2))
\\[1.5 ex]
 \leq{}&  C \int \E[ (D_{z_1}V_n)^4]^{1/2} \E[|D^2_{z_2,z_1} V_n|^4]^{1/2}\, \lambda^{ 2}(d(z_1,z_2))=C M. \label{sdlfjsgdfgs}
\end{align}

The inequalities \eqref{gamma6.1}, \eqref{sdfljsdgs}, \eqref{sdlfjsgdfgs} and  \eqref{neben} 
now imply \eqref{slfjsdlfj}, and the proof of the lemma is complete.  
\end{proof}

\begin{proof}[Proof of Theorem \ref{maintheo} for the Kolmogorov distance]
We now combine the statements of our second-order Poincar\'e inequality \eqref{sdlfjsdkgfsd}, and  Lemmas~\ref{lem:Gamma1}, \ref{lem:Gamma2}, \ref{lem:Gamma34}, \ref{lem:Gamma5} and \ref{lem:Gamma6}. For $\alpha\beta \geq 8/3$ we have the  inequality 
\begin{align*} 
d_K(V_n/v_n,Z)\leq {}& C(\gamma_1+\gamma_2+\overline{\gamma}_3 + \gamma_4+\gamma_5+\gamma_6)  
\\ \leq {}& C
\begin{cases}
n^{-1/2} & \text{if } \alpha\beta>4, \\
n^{-1/2} \log(n) \qquad  & \text{if } \alpha\beta=4, \\
n^{(2-\alpha\beta)/4} & \text{if } \frac{8}{3}\leq \alpha\beta <4.
\end{cases}
\end{align*}
On the other hand, for $2<\alpha\beta<  8/3$ we will use the bound 
\[
d_K\left( V_n /v_n, Z\right) \leq  \sqrt{d_W\left( V_n /v_n, Z\right)} 
\leq C n^{(2-\alpha\beta)/4}.
\]
Thus, we obtain the assertion of Theorem \ref{maintheo} for the Kolmogorov distance. 
\end{proof}

All our corollaries, i.e.,  Corollaries~\ref{sdlfjsdsdfsdflfj}--\ref{lsjdflhsdgdg}, follow directly from Theorem~\ref{maintheo}. 

\begin{center}
\large{\textbf{Acknowledgements}}
\end{center}
  Andreas Basse-O'Connor's research was supported by the grant DFF-4002-00003 from the Danish Council for Independent Research. Mark Podolskij gratefully acknowledges financial support through the research project  ``Ambit fields: probabilistic properties and statistical inference'' funded by Villum Fonden.

\bibliographystyle{chicago}

\end{document}